\documentclass[12pt]{amsart}
\usepackage[utf8]{inputenc}
\usepackage[lite, alphabetic, nobysame]{amsrefs}
\usepackage{amsmath,mathtools}
\usepackage{amssymb}
\usepackage{srcltx}
\usepackage{xcolor}
\usepackage{tikz}
\usepackage{xspace}
\usepackage{enumerate}
\usepackage{geometry}
\usepackage{marginnote}

\usepackage{amsmath, amsthm}
\usepackage{amssymb}
\usepackage{todonotes}

\usepackage{cite}

\usepackage{tikz}
\usetikzlibrary{positioning}

\theoremstyle{definition}

\newtheorem{thm}{Theorem}[section]
\newtheorem*{thm*}{Theorem}
\newtheorem{lemma}[thm]{Lemma}
\newtheorem{lem}[thm]{Lemma}
\newtheorem{defn}[thm]{Definition}
\newtheorem{claim}[thm]{Claim}
\newtheorem{prop}[thm]{Proposition}
\newtheorem{cor}[thm]{Corollary}

\newtheorem{remark}[thm]{Remark}

\newtheorem{question}[thm]{Question}
\newtheorem{ex}[thm]{Example}

\newtheorem{conj}[thm]{Conjecture}

\renewcommand{\subset}{\subseteq}

\newcommand\force{\Vdash}
\newcommand\R{\mathbb{R}}
\newcommand\N{\mathbb{N}}
\newcommand\Q{\mathbb{Q}}
\renewcommand\P{\mathbb{P}}
\renewcommand\L{\mathcal{L}}

\newcommand\add{\mathbf{add}(\mathcal{B})}

\DeclareMathOperator{\dom}{dom}

\newcommand{\set}[2]{ \left\{ #1 :\, #2 \right\} }
\newcommand{\seqq}[2]{ \left\langle #1 :\, #2\right\rangle }

\title[Classifying invariants for $E_1$: A tail of a generic real]{Classifying invariants for $E_1$: \\A tail of a generic real}
\author[Assaf Shani]{Assaf Shani}
\address{Department of Mathematics, Harvard University, Cambridge, MA 02138, USA}
\email{shani@math.harvard.edu}
\date{\today}

\keywords{Complete classification,
Classifying invariants,
Borel equivalence relations,
Borel reducibility,
Intersection models,
Turbulence, 
Hypersmooth equivalence relations,
The bounding number.
}

\subjclass[2010]{Primary: 03E15, 03E25, 03E17, 03E47, 03E75.}

\begin{document}
\maketitle

\begin{abstract}
Let $E$ be an analytic equivalence relation on a Polish space. We introduce a framework for studying the possible ``reasonable'' complete classifications and the complexity of possible classifying invariants for $E$, such that: (1) the standard results and intuitions regarding classifications by countable structures are preserved in this framework; (2) this framework respects Borel reducibility; (3) this framework allows for a precise study of the possible invariants of certain equivalence relations which are not classifiable by countable structures, such as $E_1$.

In this framework we show that $E_1$ can be classified, with classifying invariants which are $\kappa$-sequences of $E_0$-classes where $\kappa=\mathfrak{b}$, and it cannot be classified in such a manner if $\kappa<\add$.

These results depend on analyzing the following sub-model of a Cohen real extension, introduced in \cite{Kanovei-Sabok-Zapletal-2013} and \cite{Larson_Zapletal_2020}. Let $\seqq{c_n}{n<\omega}$ be a generic sequence of Cohen reals, and define the tail intersection model
\begin{equation*}
    M=\bigcap_{n<\omega}V[\seqq{c_m}{m\geq n}].
\end{equation*}
An analysis of reals in $M$ will provide lower bounds for the possible invariants for $E_1$.
%The question of possible invariants for $E_1$ is closely related to questions about which sets are in $M$ and which fragments of choice hold in $M$.

We also extend the characterization of turbulence from \cite{Larson_Zapletal_2020} in terms of intersection models.
\end{abstract}

\section{Introduction}\label{section: introduction}
Let $E$ be an equivalence relation on $X$. A \textbf{complete classification} of $E$ is a map $c\colon X\to I$ satisfying
\begin{equation*}
    x\mathrel{E}y\iff c(x)=c(y).
\end{equation*}
The members of $I$ are said to be \textbf{classifying invariants} for $E$.
Given an equivalence relation $E$, one would like to classify it using the simplest possible invariants. 

What constitutes a ``reasonable'' classification depends on the field and objects of study.
One famous example is the classification of compact orientable surfaces, up to homeomorphism, by their genus. The reader is referred to \cite[Preface]{Hjo00} for more examples and a thorough discussion.

We view the following requirement as a key property of a ``reasonable'' classification:
\begin{center}
$(\star)$\,    given $x\in X$, its invariant $c(x)$ is simple to define, and compute, using $x$.    
\end{center}
This should exclude the following two, not useful, ``bad'', complete classifications:
\begin{itemize}
    \item Appealing to the axiom of choice, there is a map choosing one element out of each equivalence class. This choice map is a complete classification.
    \item The map sending $x$ to its equivalence class $[x]_E=\set{y\in X}{x\mathrel{E}y}$, is a complete classification. 
\end{itemize}
In the first example, the map is not definable in any reasonable way. In particular, the invariant of a given object cannot be computed from it.
In the second, the issue is that the classifying objects are not simple to describe.
For example, given a compact orientable surface, one cannot quite simply \textit{describe all} the surfaces which are homeomorphic to it (this is an enormous collection of objects), but its genus, a natural number, is simple to describe and can be directly computed.

Our focus here is on equivalence relations defined on Polish spaces. This generally captures the situation when studying mathematical structures with some inherent separability assumption (see \cite{Kechris-DST-1995, Kechris_1997-classification-problems}).
Let $E$ and $F$ be equivalence relations on Polish spaces $X$ and $Y$ respectively. A map $f\colon X\to Y$ is a \textbf{reduction} of $E$ to $F$ if 
\begin{center}
    $x\mathrel{E}x'\iff f(x)\mathrel{F}f(x')$, for any $x,x'\in X$,
\end{center}
that is, $f$ reduces to problem determining $E$-relation to that of $F$-relation.
If $f\colon X\to Y$ is a reduction from $E$ to $F$ and $c\colon Y\to I$ is a complete classification of $F$, then the composition $c\circ f\colon X\to I$ is a complete classification of $E$, using the same set of invariants. So if there is a ``sufficiently definable'' reduction $f$ from $E$ to $F$, the invariants necessary to classify $E$ are no more complicated than those necessary to classify $F$.

$E$ is \textbf{Borel reducible} to $F$, denoted $E\leq_B F$, if there is a Borel map $f\colon X\to Y$ reducing $E$ to $F$. 
Borel reducibility is the most common notion for comparing the complexity of equivalence relations on Polish spaces, especially when the equivalence relations are Borel or analytic. In particular, Borel reducibility respects the intuitive idea of complete classification.

An equivalence relation $E$ on a Polish space $X$ is said to be \textbf{concretely classifiable} (or \textbf{smooth}) if there is a Polish space $Y$ and a Borel map $c\colon X\to Y$ which is a complete classification of $E$; equivalently, if $E$ is Borel reducible to $=_\R$, the equality relation on the reals. That is, $E$ admits real numbers as classifying invariants. (See \cite[Section 5.4]{Gao09}.)

Recall the equivalence relation $E_0$ on $2^\omega$, identifying two binary sequences $x,y\in 2^\omega$ if $\exists n\forall m\geq n(x(m)=y(m))$.
$E_0$ is the canonical obstruction for concrete classifications, according to the dichotomy theorem of Harrington, Kechris, and Louveau \cite{Harrington-Kechris-Louvaeu-E0-1990}: for any Borel equivalence relation $E$, either $E$ is concretely classifiable or $E_0\leq_B E$.

Hjorth and Kechris \cite{Hjorth-Kechris-Ulm-1995} extended the notion of concrete classifications to Ulm classifications, where the classifying invariants are in $2^{<\omega_1}$, countable subsets of $\omega_1$. They also extended the Harrington-Kechris-Louveau dichotomy: for any analytic equivalence relation $E$, either $E$ is Ulm classifiable or $E_0\leq_B E$ (assuming the existence of sharps for reals).

Given a countable first order language $\L$, let $X_\L$ be the Polish space of all countable $\L$-structures on the set $\omega=\{0,1,2,...\}$, and let $\cong_\L$ be the isomorphism equivalence relation on $X_\L$ (see \cite[Chapter 11]{Gao09}). 
An equivalence relation $E$ on a Polish space $X$ is \textbf{classifiable by countable structures} if it is Borel reducible to $\cong_\L$ for some $\L$ (see \cite{HK96, Kechris_1997-classification-problems, Hjo00}).
In this case, the $E$-classes can be classified using countable structures, such as countable groups or graphs, up to isomorphism. Moreover, for the isomorphism relation, the Scott analysis provides a complete classification with hereditarily countable sets as classifying invariants (see \cite[Chapter 12]{Gao09}).

For an equivalence relation $E$ on $X$, its {\bf Friedman-Stanley jump}, $E^+$, is defined on $X^\N$ by $x\mathrel{E^+}y\iff \forall n\exists m (x(n)\mathrel{E}y(m))$ and $\forall n\exists m(y(n)\mathrel{E}x(m))$. For example, the map sending $x\in\R^\N$ to $\set{x(n)}{n\in\N}$ is a complete classification of $=_\R^+$. An equivalence relation $E$ is considered ``classifiable by countable sets of reals'' if it is Borel reducible to $=_\R^+$. Similarly, $E$ is considered ``classifiable by hereditarily countable sets of sets of reals'' if it is Borel reducible to $(=_\R^{+})^+$, and so on. This approach is taken in \cite[Introduction (E)]{HK96}.

Given a continuous action $a\colon G\curvearrowright X$ of a Polish group $G$ on a Polish space $X$, let $E_a$ be the induced {\bf orbit equivalence relation}, $x\mathrel{E_a}y\iff \exists g\in G (g\cdot x=y)$.
Hjorth's turbulence property provides a condition on the action so that $E_a$ is not classifiable by countable structures (see \cite{Hjo00,Kechris_1997-classification-problems}).
Among orbit equivalence relations, Hjorth \cite{Hjorth-turb-dichotomy-2002} showed that turbulence is the canonical obstruction: for a Borel orbit equivalence relation $E_a$, either $E_a$ is classifiable by countable structures or there is a turbulent action $a'$ such that $E_{a'}\leq_B E_a$. 

There are interesting non-orbit Borel equivalence relations. These are referred to as ``the dark side'' in \cite[Chapter 8]{Hjo00}.
The equivalence relation $E_1$ is defined on $\mathbb{R}^\N$ by $x\mathrel{E_1}y\iff \exists n \forall m>n (x(m)=y(m))$. Kechris and Louveau \cite{Kechris-Louveau-1997}, extending \cite{Kechris-ctbl-sections-1992}, proved that $E_1$ is not Borel reducible to any orbit equivalence relation induced by a Polish group action. (In particular, $E_1$ is not classifiable by countable structures.)
%An equivalence relation $E$ is \textbf{hypersmooth} if it can be written as an increasing union of smooth equivalence relations. Kechris and Louveau \cite{Kechris-Louveau-1997} classified the hypersmooth equivalence relations, up to Borel bireducibility: these are either $E_1$, $E_0$, or are smooth.  

Some natural classification problems in mathematics lie in ``the dark side''.
Solecki \cite{Solecki-2002-composants} showed that for a hereditarily indecomposable continuum, the equivalence relation partitioning it into its composants is Borel bireducible with $E_1$.
Thomas showed that $E_1$ is Borel bireducible with the quasi-equality relation on finitely generated groups \cite[Theorem 5.6]{Thomas-2008-quasi-iso}, and that the equivalence relation of virtual isomorphism on finitely generated groups, $\approx_{\mathrm{VI}}$, is Borel reducible to $E_1^+$ \cite[Theorem 6.8]{Thomas-2008-quasi-iso}. 

This paper suggests a framework of \textbf{generically absolute classifications}, Definition~\ref{defn;generic-classification}, in which equivalence relations such as $E_1$ and $E_1^+$ admit ``reasonable classifying invariants''.

%Given an equivalence relation $E$ on $X$, its Friedman-Stanley jump, denoted $E^+$, is defined on $X^\omega$ by
%\begin{equation*}    x\mathrel{E^+}y\iff \forall n \exists m (x(n)\mathrel{E}y(m)) \textrm{ and }\forall n \exists m (y(n)\mathrel{E}x(m)). \end{equation*}
%Starting from the equality relation on the reals, $=_\R$, the Friedman-Stanley hierarchy is defined inductively by $=^{+\alpha+1}= (=^{+\alpha})^+$, and $=^{+\beta}=\prod_{\alpha<\beta}=^{+\alpha}$ for limit ordinals $\beta$.

%Note that $=^+$ admits the complete classification $c\colon\mathbb{R}^\omega\to\mathcal{P}_{\aleph_0}(\mathbb{R})$ defined by $x\mapsto\set{x(i)}{i\in\omega}$. This equivalence relation is often referred to as ``equality on countable sets of reals''. Similarly, the second Friedman-Stanley jump, $=^{++}$, can be seen as ``equality on hereditarily countable sets of sets of reals'', as so on. A Borel equivalence relation is classifiable by countable structures if and only if it is Borel reducible to $=^{+\alpha}$ for some countable ordinal $\alpha$ (see \cite{HKL_1998} or \cite{Friedman2000}.)

\subsection{Abstract complete classification}\label{sec;complete-class}
Let $E$ be an analytic equivalence relation on a Polish space $X$. (More generally, we just need that $E$ has an absolute definition.) 
We consider the following three conditions as the key properties making a complete classification $c$ ``reasonable'', ensuring the informal requirement $(\star)$ above.
A good example to keep in mind is $=_{\mathbb{R}}^+$, with the ``reasonable'' classifying map $\mathbb{R}^\N\to\mathcal{P}_{\aleph_0}(\mathbb{R})$ defined by $x\mapsto\set{x(i)}{i\in\N}$, and the ``bad'' classifying map $x\mapsto [x]_{=^+}$.
\begin{enumerate}
    \item{[Definability]} There is a set theoretic formula $\psi$ and a parameter $b$ such that $c(x)=A\iff \psi(x,A,b)$.
    \item{[Locality]} If $V\subset W$ are ZF models, $x,A\in V$ and $\psi^V(x,A,b)$, then $\psi^W(x,A,b)$.
    \item{[Absoluteness]} If $W$ is a ZF extension of $V$ then $c$ (defined using $\psi$) is a complete classification of $E$ in $W$ as well.
    %\item{[Recovery]} Given $A=c(y)$ one can ``easily'' recover some $y$ such that $yEx$ ($c(x)=A$). This generally requires some form of choice, e.g. to enumerate $A$.
\end{enumerate}
In this case, say that $c$ is an \textbf{absolute complete classification}.
We consider Locality to be the most fundamental necessary property. It tells us that given $x$ one can actually calculate/construct the invariant $c(x)$ in a concrete way. For example, it prohibits the classification $x\mapsto[x]_E$, when $E$ is a not-countable Borel equivalence relation.

These three properties seem to capture the key features of classification by countable structures.
\begin{ex}
For an isomorphism relation, the Scott analysis satisfies these properties. This is precisely the point of view taken in \cite{Friedman2000} and \cite{Ulrich-Rast-Laskowski-2017}.
\end{ex}
\begin{ex}
The Ulm-type classifications of Hjorth and Kechris \cite{Hjorth-Kechris-Ulm-1995} are absolute. 
%These capture the notion of having classifying invariants in the space $2^{<\omega_1}$ of all countable subsets of $\omega_1$.
Moreover, for the specific case of $I=2^{<\omega_1}$, and when the parameter $b$ is a member of a Polish space, absolute classifications are essentially the same as Ulm classifications.
\end{ex}
\begin{ex}
More generally, if $I=\mathrm{HC}$, the space of hereditarily countable sets, and the parameter $b$ is a member of a Polish space, then the notion of absolute classification essentially agrees with the $\mathbf{\Delta}^1_2$-absolute classification of \cite[Chapter 9]{Hjo00}.
\end{ex}

\begin{remark}
Some sort of Absoluteness is necessary for a reasonable notion of classification. For example, in the constructible universe $L$, any equivalence relation $E$ admits a definable classification $x\mapsto ([x]_E)^L$ satisfying (1) and (2), but fails (3) whenever new $E$-classes are added. 
\end{remark}

To extend this notion beyond classifications by countable structures, we relax the Absoluteness requirement (3), to only some forcing extensions.
In the context of equivalence relations defined on Polish spaces, the most natural structure provided is the topology, so it is natural to require absoluteness for Cohen forcing. That is, to require the classification ``respects Baire-category arguments''.

\begin{defn}\label{defn;generic-classification}
Let $E$ be an analytic equivalence relation on a Polish space $X$.
Say that a complete classification $c\colon X\to I$ is a \textbf{generically absolute classification} if
\begin{enumerate}
    \item{[Definability]} There is a set theoretic formula $\psi$ and a parameter $b$ such that $c(x)=A\iff \psi(x,A,b)$.
    \item{[Locality]}  If $V\subset W$ are ZF models and $x,A\in V$ with $\psi^V(x,A,b)$, then $\psi^W(x,A,b)$.
    \item 
    \begin{enumerate}
        \item {[Invariance absoluteness]} If $W$ is a ZF extension of $V$ then $\psi$ defines an $E$-invariant map in $W$. (That is, the map $c$ defined from $\psi$ is a well defined function and is $E$-invariant: $x\mathrel{E}y\implies c(x)=c(y)$.)
        \item {[Generic absoluteness]} If $W$ is a generic extension of $V$ by a single Cohen real, then $\psi$ defines a complete classification of $E$ in $W$.
    \end{enumerate}    
\end{enumerate}
If such a complete classification exists, say that $E$ is \textbf{generically classifiable}.
\end{defn}

Before going further, let us stress that this notion respects and extends the usual intuitions regarding classification.

\begin{remark}\label{remark: gen class respects reducibility}
If $E\leq_B F$ and $F$ admits a generically absolute classification with classifying invariants in $I$, then so does $E$ (Proposition~\ref{prop; Apdx borel reduction respects class}).
\end{remark}
\begin{ex}\label{example: gen class =R no ord inv}
The equality relation $=_\R$ does not admit a generically absolute classification with ordinal invariants (Proposition~\ref{prop: Apdx =_R not ordinal classifiable}). It then follows from Burgess' theorem \cite[35.21 (ii)]{Kechris-DST-1995} that for an analytic equivalence relation $E$, the following are equivalent: $E$ is absolutely classifiable with countable ordinal invariants; $E$ is generically classifiable with ordinal invariants; ${=_\R} \not\leq_B {E}$.
%The equivalence of the last two items is, in a sense, analogous to Woodin's dichotomy in $L(\R)$ \cite[Theorem 4.3]{Hjorth-handbook}.
\end{ex}
\begin{ex}\label{example: gen class E0 no power of ordinal invs}
$E_0$ does not admit a generically absolute classification with invariants in $2^\alpha$ for any ordinal $\alpha$ (Proposition~\ref{prop; Apdx E0 not P(ord) class}). It then follows from the dichotomy of Hjorth and Kechris \cite[Theorem 1]{Hjorth-Kechris-Ulm-1995} that (assuming the existence of sharps for reals) for an analytic equivalence relation $E$, the following are equivalent: $E$ is absolutely classifiable with invariants in $2^{<\omega_1}$; $E$ is generically classifiable with invariants in $2^\alpha$ for some ordinal $\alpha$; $E_0\not\leq_B E$. 
%The equivalence of the last two items is, in a sense, analogous to Hjorth's dichotomy in $L(\R)$ \cite[Theorem 4.4]{Hjorth-handbook}.
\end{ex}
\begin{ex}\label{example: gen class turb no invs}
If $E$ is generically turbulent, then $E$ does not admit a generically absolute classification, with any set of invariants. This follows from the characterization of generic turbulence in \cite[Theorem 3.2.2]{Larson_Zapletal_2020}, see Section~\ref{section: turbulence}.
It then follows from Hjorth's dichotomy \cite[Theorem 9.7]{Hjo00} that for an orbit equivalence relation $E_a$, the following are equivalent: $E_a$ admits a generically absolute classification; $E_a$ admits an absolute classification with hereditarily countable invariants; there is no turbulent action $a'$ with $E_{a'}\leq_B E_a$.
\end{ex}

The following is the main result of this paper. First, $E_1$ admits ``reasonable'' invariants. More importantly, we prove lower bounds, suggesting this is a non trivial notion.
\begin{thm}\label{thm;E_1-classification}
\begin{enumerate}
    \item There is a generic classification of $E_1$ with classifying invariants that are sequences of $E_0$-classes of length $\mathfrak{b}$;
    \item there is no generic classification of $E_1$ using classifying invariants which are sequences of $E_0$-classes of length $<\add$;
    \item there is no complete classification of $E_1$ which is absolute for all forcings.
\end{enumerate}
\end{thm}
Here $\mathfrak{b}$ is the bounding number, and $\add$ is the additivity number of the meager ideal. The reader is referred to \cite{Blass-handbook-2010} for their definitions and more background.
Part (1) is proven in Section~\ref{sec: complete class for E1}. The lower bounds, parts (2) and (3), are proven in Section~\ref{sec;intersection-model}. 
%They follow from the analysis of the intersection model in Section~\ref{sec;intersection-model}.
\begin{question}
What is the minimal $\kappa$ such that $E_1$ admits a generic classification with invariants which are $\kappa$-sequences of $E_0$-classes?
\end{question}
By the above theorem, $\add\leq\kappa\leq\mathfrak{b}$.
We note that the gap between $\mathfrak{b}$ and $\add$ is related to the result of Chichon and Pawlikowsky \cite{Chichon-Pawlikowski-1986}, that adding a single Cohen real collapses the bounding number $\mathfrak{b}$ in the generic extension to be $\add^V$, the additivity number as computed in the ground model.

If $E$ can be generically classified with invariants in $I$, then its Friedman-Stanley jump $E^+$ can be generically classified, with countable subsets of $I$ as classifying invariants (Proposition~\ref{prop: jump is classifiable}).
In particular, $E_1^+$ is generically classifiable.
By the above mentioned result of Thomas, $\approx_{\mathrm{VI}}{\leq_B} {E_1^+}$, and therefore $\approx_{\mathrm{VI}}$ is generically classifiable. (It is not classifiable by countable structures, as $E_1{\leq_B} \approx_{\mathrm{VI}}$.)
\begin{question}
What are the optimal classifying invariants for a generic classification of $\approx_{\mathrm{VI}}$, the virtual isomorphism relation on finitely generated groups?
\end{question}

We conclude with an ambitious conjecture: when replacing ``classification by countable structures'' with ``generic classification'', Hjorth's dichotomy for turbulence \cite{Hjorth-turb-dichotomy-2002} can be extended to non-orbit equivalence relations.
We state this conjecture with a generalized definition of turbulence, for non-orbit equivalence relations (Definition~\ref{defn: generically turbulent double brackets}). This definition is motivated by the  characterization of turbulence of Larson and Zapletal \cite[Theorem 3.2.2]{Larson_Zapletal_2020}, and an extension of this characterization, Theorem~\ref{thm: turbulence characterization intersection model} below.

\begin{conj}
Let $E$ be an analytic equivalence relation. Then exactly one of the following hold.
\begin{itemize}
    \item A turbulent equivalence relation (Definition~\ref{defn: generically turbulent double brackets}) is Borel reducible to $E$.
    \item $E$ admits a generically absolute classification.
\end{itemize}
\end{conj}

\subsection{An intersection model}\label{sec: intro - intersection model}
Let $\mathbb{P}$ be Cohen forcing on $\R^\omega$. Fix $x=\seqq{x(n)}{n<\omega}$ a $\mathbb{P}$-generic over $V$, and let $x_m=0^m\, ^\frown \seqq{x(n)}{n>m}$, a sequence of $m$ zeros followed by a tail of $x$. Define the intersection model 
\begin{equation*}
    M=\bigcap_{n<\omega}V[x_n].
\end{equation*}
In a sense, $M$ sees all ``generic tail events''. Its relationship with the equivalence relation $E_1$ is immediate: each $x_m$ is $E_1$-related to $x$.
This model and its relationship with $E_1$ was introduced in \cite[4.4]{Kanovei-Sabok-Zapletal-2013} and \cite[4.1]{Larson_Zapletal_2020}.
%Note that $M=V[[x]]_{E_1}$ (see \cite{KSZ13}). Indeed, it follows from the definition that $V[[x]]_{E_1}$ is included in $V[x_n]$ for each $n$, and therefore is included in $M$. On the other hand, for any $y\mathrel{E_1}x$, in some generic extension, there is some $n$ such that $V[x_n]\subset V[y]$. Thus $M$ is included in $V[y]$ for any such $y$, and so $M$ is included in $V[[x]]_{E_1}$.
In both cases, interesting properties of $E_1$ were deduced from simple properties of this intersection model: essentially that each $x(n)$ is not in $M$.
It follows from \cite[3.1.1 (i)]{Kanovei-Sabok-Zapletal-2013}, or \cite[4.2.9]{Larson_Zapletal_2020}, that $M$ is a model of ZF.
However, the precise identity of $M$ remained open, even whether or not it satisfies ZFC.

These questions about $M$ are closely related to questions about possible generic classifications of $E_1$. For example, we show that the axiom of choice in fact fails in $M$, and this is closely related to the generic classification of $E_1$ from Theorem~\ref{thm;E_1-classification}.

\begin{thm}\label{thm;intersection-model}
\begin{enumerate}
    \item ``Choice for $\mathfrak{b}$-sequences of $E_0$-classes'' fails in $M$;
    \item $M$ satisfies DC;    
    \item $M=V(B)$, the minimal ZF-model extending $V$ and containing $B$, where $B$ is a set of reals.
\end{enumerate}
\end{thm}
(Note that $\mathfrak{b}$ as calculated in $M$ is in fact $\add^V$, by the aforemetioned result of Cichon and Pawlikowsky.) 
We also establish a characterization of reals in $M$, see Corollary~\ref{cor: reals in M}, which is closely related to the proof of part (2) in Theorem~\ref{thm;E_1-classification}.
Part (3) of Theorem~\ref{thm;intersection-model} can be seen as a weak positive fragment of choice in $M$, and will be used to conclude part (3) of Theorem~\ref{thm;E_1-classification}.

\begin{question}
What other forms of choice hold in the intersection model $M$?
Does $\mathrm{DC}_{\kappa}$ hold for an uncountable $\kappa<\add^V$?
\end{question}

\begin{question}
Suppose $x\in[0,1]^\omega$ is a Random real, with respect to the product measure. Let $N=\bigcap_{n<\omega}V[x_n]$, where $x_n$ is the tail of $x$ past $n$. What can we say about $N$?
What reals in $N$ look like? 
%Is $N$ of the form $V(A)$ for some set $A$?
(As in Proposition~\ref{prop;DC}, $N$ is a model of DC.)
%[-- Probably $N$ is not a model of choice by similar arguments as above? --]
\end{question}

\subsection*{Acknowledgments}
I would like to thank James Cummings for numerous long and insightful conversations, without which this paper would not have been.
I would also like to thank Filippo Calderoni, Clinton Conley, Paul Larson, and Jindrich Zapletal for very helpful discussions. 

\section{A complete classification for $E_1$}\label{sec: complete class for E1}

Replace $\R$ with the space $2^\omega$, and consider $E_1$ as an equivalence relation on $(2^\omega)^\omega$ (see \cite[8.1]{Gao09}.)
Given $x\in(2^\omega)^\omega$, we consider $x$ as a function $x\colon\omega\times\omega\to\{0,1\}$, a matrix of 0's and 1's, and identify it with a member of the space $2^{\omega\times\omega}$.
So for $x,y\in (2^\omega)^\omega$, they are $E_1$-equivalent if the agree on all but finitely many columns.

We begin by identifying some interesting $E_1$-invariants. We will reach a complete classification of $E_1$ by collecting ``enough'' of these invariants.

\begin{defn}
Given a function $f\colon\omega\to\omega+1$, define the function $x\restriction f$ whose domain is $\set{(n,m)}{m<f(n)}$, the area below the graph of $f$, by
\begin{equation*}
    (x\restriction f)(n,m)= x(n,m)\hspace{0.2cm}\textrm{for }m<f(n).
\end{equation*}
\end{defn}

\begin{defn}
Given a partial function $y\colon \omega\times\omega\to\{0,1\}$, define $[y]$ to be the set of all functions $z$ with the same domain as $y$ such that $z$ and $y$ differ in at most finitely many places.
\end{defn}
Note that $[y]$ can be identified with an $E_0$-class.
\begin{remark}
Let $f\colon\omega\to\omega$ be some function and suppose that $y,z\colon\omega\times\omega\to\{0,1\}$ are partial functions with domain $\set{(n,m)}{m<f(n)}$. Then $[y]=[z]$ if and only if $y$ and $z$ agree on all but finitely many columns.
\end{remark}
In conclusion:
\begin{claim}
Given $f\colon\omega\to\omega$, the map $x\mapsto [x\restriction f]$ is $E_1$-invariant. That is, if $x,y$ are $E_1$-related, then $[x\restriction f]=[y\restriction f]$.
\end{claim}

The reader is refered to either \cite{Blass-handbook-2010, Halbeisen-book-2017, Jech2003} for the eventual domination order $<^\ast$ on $\omega^\omega$, the space of functions from $\omega$ to $\omega$.
Fix $\seqq{f_\alpha}{\alpha<\mathfrak{b}}$ an $<^\ast$-increasing and unbounded sequence of functions in $\omega^\omega$ such that each $f_\alpha$ is increasing. For $x\in (2^\omega)^\omega$, define
\begin{equation*}
    c(x)=\seqq{[x\restriction f_\alpha]}{\alpha<\mathfrak{b}}.
\end{equation*}
\begin{claim}\label{claim: complete classification of E1}
$c$ is a complete classification of $E_1$.
\end{claim}
\begin{proof}
The map is invariant, as each coordinate is invariant. It remains to show that if $c(x)=c(y)$ then $x\mathrel{E_1} y$.
In fact it suffices to assume that for unboundedly many $\alpha<\mathfrak{b}$, $[x\restriction f_\alpha]=[y\restriction f_\alpha]$.
If so, we many find an unbounded subset $X\subset\mathfrak{b}$ and some $k<\omega$ such that for any $\alpha\in X$, any $n>k$ and $m<f_\alpha(n)$, $(x\restriction f_\alpha)(n,m)=(y\restriction f_\alpha)(n,m)$.
That is, 
\begin{center}
$(\star)$\, for any $\alpha\in X$, any $n>k$ and $m<f_\alpha(n)$, $x(n,m)=y(n,m)$.    
\end{center}
As $\seqq{f_\alpha}{\alpha\in X}$ is $<^\ast$-unbounded, there must be some $k'\geq k$ such that $\set{f_\alpha(k')}{\alpha\in X}$ is unbounded in $\omega$. Since each $f_\alpha$ is increasing, it follows that for any $n\geq k'$, $\set{f_\alpha(n)}{\alpha\in X}$ is unbounded in $\omega$.
Finally, given any $n\geq k'$ and any $m$, there is some $\alpha\in X$ such that $f_\alpha(n)>m$, and so, by $(\star)$, $x(n,m)=y(n,m)$. That is, $x$ and $y$ agree on all columns past $k'$, and so are $E_1$-equivalent.
\end{proof}
Note that the map $x\mapsto c(x)$ is definable using the parameter $\seqq{f_\alpha}{\alpha<\mathfrak{b}}$.
Furthermore the calculation of $c(x)$ is absolute (local), that is, clause (2) of Definition~\ref{defn;generic-classification} holds.
The map $c$ is $E_1$-invariant in any extension, so clause (3)(a) of Definition~\ref{defn;generic-classification} is satisfied.
Finally, in the proof above we only required that $\seqq{f_\alpha}{\alpha<\mathfrak{b}}$ is unbounded in the eventual domination order, to conclude that $c$ is a complete classification. Since Cohen forcing does not add dominating reals (see \cite[Lemma~22.2]{Halbeisen-book-2017}), $\seqq{f_\alpha}{\alpha<\mathfrak{b}}$ remains unbounded in a Cohen-real extension, so the map $c$ remains a complete classification, thus (3)(b) of Definition~\ref{defn;generic-classification} is satisfied.
This concludes part (1) of Theorem~\ref{thm;E_1-classification}.

\section{The intersection model for $E_1$}\label{sec;intersection-model}
Let $\mathbb{P}$ be Cohen forcing on $2^{\omega\times\omega}$. That is, the conditions of $\P$ are finite partial functions from $\omega\times\omega$ to $\{0,1\}$, ordered by reverse extension. A generic filter for $\P$ can be identified with a generic real $x\in 2^{\omega\times\omega}$. Note that $\P$ is forcing equivalent to Cohen forcing for adding one real in $2^\omega$.

Fix $x\in 2^{\omega\times\omega}$ a $\mathbb{P}$-generic over $V$.
Define the step functions $\delta_n\colon \omega\to \omega+1$ by $\delta_n(m)=0$ if $m<n$ and $\delta_n(m)=\omega$ if $m\geq n$, and let $x_n=x\restriction\delta_n$. That is, $x_n$ is a tail of columns of $x$. 
Let $\mathbb{P}^n$ be the poset of all finite partial functions from $(\omega\setminus n)\times\omega$ to $\{0,1\}$. So $\mathbb{P}=\mathbb{P}^0$, and each $x_n$ may be identified as a $\mathbb{P}^n$-generic.

We are interested in the following intersection model, which we fix for the remainder of the paper.
\begin{defn}\label{defn: intersection model M}
$M=\bigcap_{n<\omega}V[x_n]$.
\end{defn}
This is the same model as described in Section~\ref{sec: intro - intersection model} (after identifying $\R^\omega$ with $2^{\omega\times\omega}$). $M$ satisfies ZF by \cite[4.2.9]{Larson_Zapletal_2020}.
Moreover, $M$ satisfies DC (part (2) of Theorem~\ref{thm;intersection-model}).
\begin{prop}\label{prop;DC}
DC holds in $M$.
\end{prop}
\begin{proof}
Suppose $R$ is a relation on $X$, in $M$, such that $\forall x\exists y (x\mathrel{R}y)$ holds.
So $X$ is in $V[x_n]$ for each $n$.

For each $n$ we may choose a well ordering of $X$ in $V[x_n]$, definably in $X$ and $x_n$, as follows: choose a minimal $\P^n$-name $\tau$ (according to some well ordering in $V$) so that $\tau[x_n]$ is a well ordering of $X$. This definition can be carried uniformly in any model $v[x_m]$ with $m\leq n$.

Define a sequence $\seqq{x_n}{n<\omega}$ recursively so that $x_{n+1}$ is the minimal, according to the above chosen well ordering of $X$ in $V[x_{n+1}]$, with $x_n\mathrel{R}x_{n+1}$.
For each $m$, the sequence $\seqq{x_n}{n\geq m}$ is definable in $V[x_m]$.
Each $x_i$ is in $M$, and therefore in $V[x_m]$, and so the sequence $\seqq{x_n}{n<\omega}$ is in $V[x_m]$, for each $m$. It follows that $\seqq{x_n}{n<\omega}\in M$, as required.
\end{proof}
\begin{remark}
More generally, DC holds in the intersection model arising from any countable coherent sequence: see \cite{Larson_Zapletal_2020}, ``Supplemental materials'', ``Updated Theorem 4.2.9''.
\end{remark}

Next we prove part (1) of Theorem~\ref{thm;intersection-model}, that the axiom of choice fails in $M$.

\begin{remark}\label{remark: stuff in M}
Thinking of $x$ as a matrix of 0's and 1's, one can see that each row of $x$, as an element of $2^\omega$, is in $M$. In fact, given $f\in\omega^\omega$ in $V$, $x\restriction f$ is in $M$.
More generally, given $f\in \omega^\omega$ in $M$, $x\restriction f$ is in $M$. We will see below that any real in $M$ is definable from one of those.
\end{remark}

Fix a $<^\ast$-increasing and unbounded sequence $\seqq{f_\alpha}{\alpha<\kappa}$ in $V$.
Since Cohen forcing does not add dominating functions, and $V\subset M\subset V[x]$, the sequence $\seqq{f_\alpha}{\alpha<\kappa}$ is $<^\ast$-unbounded in $M$ as well.

By the remark above, each $x\restriction f_\alpha$ is in $M$, and so is $A_\alpha=[x\restriction f_\alpha]$, the set of all finite changes of $x\restriction f_\alpha$. Define $A=\seqq{A_\alpha}{\alpha<\kappa}$. 
\begin{claim}
$A\in M$.
\end{claim}
\begin{proof}
Let $y_n$ be defined as $x$ from the $n$'th column onward, and be all zeros in the first $n$ columns. Then $y_n\in V[x_n]$ and $A_\alpha=[y_n\restriction f_\alpha]$.
We conclude that $A$ is in $V[x_n]$, for each $n$, and therefore $A$ is in $M$.
\end{proof}
\begin{lem}\label{lem;no-choice-b-E0-classes}
There is no choice function for $A$ in $M$.
\end{lem}
\begin{proof}
Assume towards a contradiction that $\seqq{a_\alpha}{\alpha<\kappa}\in M$ where $a_\alpha\in A_\alpha$. Then $[a_\alpha]=[x\restriction f_\alpha]$ for any $\alpha<\kappa$.
Working in $V[x]$, there is an unbounded $X\subset\kappa$ and $k<\omega$ such that for all $n>k$, $\alpha\in X$ and $m<f_\alpha(n)$, $x(n,m)=a_\alpha(n,m)$.

We can now find in the ground model an unbounded $Y\subset\kappa$ and $\set{p_\alpha}{\alpha\in Y}$ such that $p_\alpha\force \alpha\in\dot{X}$. 
%That is, $p_\alpha$ forces that $\dot{a}_\alpha$ and $\dot{x}\restriction f_\alpha$ agree past $k$.
By thinning out $Y$, we may assume that there is a fixed $p\in\mathbb{P}$ such that $p_\alpha=p$ for all $\alpha\in Y$.

As in the proof of Claim~\ref{claim: complete classification of E1}, fix $n>k$ such that $\set{f_\alpha(n)}{\alpha\in Y}$ is unbounded, and $\set{(n,m)}{m<\omega}$ is disjoint from the domain of $p$ and.
Let $\dot{a}^l$ be a $\mathbb{P}^{l}$-name such that $\dot{a}^l[x_{l}]=\seqq{a_\alpha}{\alpha<\kappa}$, for $l\in\{0,n+1\}$.
Take $q\in\mathbb{P}$ extending $p$ and forcing that $\dot{a}^0[\dot{x}]=\dot{a}^{n+1}[\dot{x}_{n+1}]$.
Fix some $m$ such that $(n,m)$ is not in the domain of $q$, and $\alpha\in Y$ such that $f_\alpha(n)>m$.
Let $r$ be a condition in $\mathbb{P}^{n+1}$ deciding $\dot{a}^{n+1}_\alpha(n,m)$ which is compatible with $q$.
Now the condition $r\cup q$ decides $x(n,m)=a_\alpha(n,m)$, yet $(n,m)$ is not in the domain of $r\cup q$, a contradiction.
\end{proof}

%Let us mention that $V(A)$ is strictly contained in $M$. As mentioned before, $M$ contains reals which are Cohen generic over $V$, such as the rows of $x$. Let $f\in\omega^\omega$ be in $M$ which is not bounded by any ground model function. Then $x\restriction f$ is in $M$, but is not in $V(A)$.

We now go on to prove parts (2) and (3) of Theorem~\ref{thm;E_1-classification}.

\begin{defn}
Given a sequence $\seqq{p_n}{n<\omega}$ such that $p_n$ is a condition in $\mathbb{P}^n$, the \textbf{diagonal} of the sequence $\seqq{p_n}{n<\omega}$ is the function $d\in\omega^\omega$ defined by $d(n)=$ the maximal $l$ such that $p_m(n,l)$ is defined, for some $m\leq n$.
%Note that $p_n\sqsubseteq d$ for each $n$.
\end{defn}
\begin{claim}
If $\seqq{p_n}{n<\omega}$ is in $M$, then the diagonal $d$ is also in $M$.
\end{claim}

%\begin{defn}For a partial function $y\colon \omega\to\omega\to\{0,1\}$, say that $y$ \textbf{forces} a statement $\phi$ if there is a condition $p$ forcing $\phi$ such that $y$ extends $p$.Equivalently, any large enough initial segment of $y$ forces $\phi$.\end{defn}
We write $y\sqsubset 2^{\omega\times\omega}$ to mean ``$y$ is a partial function $\omega\times\omega\to\{0,1\}$''.
Given $y,z\sqsubset 2^{\omega\times\omega}$, write $z\sqsubseteq y$ if $y$ extends $z$.

\begin{defn}
For $y\sqsubset 2^{\omega\times\omega}$ say that $y$ forces $\psi$, $y\force \psi$, if there is some $p\sqsubseteq y$ such that $p\force\psi$.
\end{defn}
The following two are the key lemmas towards understanding the reals, and other sets in $M$.
\begin{lem}\label{lem;d^Z}
Suppose $Z\in M$, $Z=\dot{Z}[x]$.
Then there is a sequence $\seqq{\tau^{\dot{Z}}_n}{n<\omega}\in M$, $\tau_0^{\dot{Z}}=\dot{Z}$, where $\tau^{\dot{Z}}_n$ is a $\mathbb{P}^n$-name such that $\tau^{\dot{Z}}_n[x_n]=Z$ and there is a function $d^{\dot{Z}}\in\omega^\omega$ in $M$ such that for each $n$,
$x\restriction d^{\dot{Z}}$ force that $\dot{Z}[\dot{x}]=\tau^{\dot{Z}}_n[\dot{x}_n]$.
\end{lem}
\begin{proof}
Working in $V[x]$, define the sequence $\seqq{\tau_n}{n<\omega}$ such that $\tau_n$ is the least $\mathbb{P}^n$-name, according to a fixed well order in $V$, with $\tau_n[x_n]=Z$. Note that $\seqq{\tau_n}{n\geq m}$ is definable in $V[x_m]$ using $Z$ and $x_m$. It follows that $\seqq{\tau_n}{n<\omega}$ is in $V[x_m]$, for each $m$, and therefore $\seqq{\tau_n}{n<\omega}$ is in $M$. Without loss of generality $\tau_0=\dot{Z}$.

For each $n$, let $q_n$ be minimal condition in $\mathbb{P}^n$ (according to some fixed well order) such that $q_n\sqsubset x_n$ and  $q_n\force\tau_n[\dot{x}_n]=\tau_{n+1}[\dot{x}_{n+1}]$. Note that $\seqq{q_n}{n\geq m}$ is definable in $V[x_m]$ from $x_m$ and $Z$.
It follows that $\seqq{q_n}{n<\omega}$ is in $M$.
Let $d$ be the diagonal of $\seqq{q_n}{n<\omega}$.
For any $n$, large enough initial segments of $x\restriction d$ force that $\tau_0$ and $\tau_n$ agree. So $d=d^{\dot{Z}}$ and $\tau^{\dot{Z}}_n=\tau_n$ are as desired.
\end{proof}

\begin{lemma}\label{lemma: restriction to dominating function}
Let $Z$ be a set in $M$, and fix a name $\dot{Z}$ for $Z$. 
Given $z=\dot{z}[x]\in M$, there is a function $f^{\dot{z},\dot{Z}}$ in $M$ such that for any $g$, if $g \geq^\ast f^{\dot{z},\dot{Z}}$ and $g\geq d^{\dot{Z}},d^{\dot{z}}$ then \begin{equation*}
    z\in Z\implies x\restriction g\force \dot{z}\in \dot{Z},\, z\notin Z\implies x\restriction g\force \dot{z}\notin\dot{Z}.
\end{equation*}
\end{lemma}
By $g\geq d$ we mean that $g(n)\geq d(n)$ for each $n$.
\begin{proof}
Fix $\seqq{\tau_n^{\dot{Z}}}{n<\omega}$, $d^{\dot{Z}}$ and $\seqq{\tau_n^{\dot{z}}}{n<\omega}$, $d^{\dot{z}}$  in $M$ as above.
Let $\seqq{p_n}{n<\omega}$ be in $M$ such that $p_n\in\mathbb{P}^n$ is minimal deciding $\tau_n^{\dot{z}}\in\tau^{\dot{Z}}_n$, $p_n\sqsubset x_n$.
Let $f^{\dot{z},\dot{Z}}\in M$ be the diagonal of $\seqq{p_n}{n<\omega}$.
Given $g$ such that $g\geq^\ast f^{\dot{z},\dot{Z}}$ and $g\geq d^{\dot{Z}},d^{\dot{z}}$, fix $m$ such that $g$ is above $f^{\dot{z},\dot{Z}}$ beyond $m$.
Then for any $n\geq m$, $x\restriction g$ decides correctly $\tau_n^{\dot{z}}\in\tau_n^{\dot{Z}}$.
In addition, $x\restriction g$ extends $x\restriction d^{\dot{Z}}$ and $x\restriction d^{\dot{z}}$, and so forces that $\tau^{\dot{Z}}_0=\tau_n^{\dot{Z}}$ and $\tau^{\dot{z}}_0=\tau_n^{\dot{z}}$.
Therefore $x\restriction g$ decides correctly the statement $\dot{z}\in\dot{Z}$.
\end{proof}
%Note that $d^Z$ and $f^{v,Z}$ depend on the choice of name $\dot{Z}$. This will not pose a problem. To define $Z$ from $x\restriction f$ we merely require that there is some $\dot{Z}$ for which the above is satisfied. Furthermore, as the names are all in $V$ there is no problem in choosing one.
The lemma will be applied first to sets $Z$ which are subsets of $V$. In this case we only consider $z\in V$, and use their canonical names $\check{z}$.

\begin{defn}\label{defn;defining-subsets-of-V}
Suppose $Z\subset U$ is in $M$, with $U\in V$, and $y\sqsubseteq x$. 
Say that $\mathbf{y}$ \textbf{defines Z via} $\mathbf{\dot{Z}}$ if for any $v\in U$, if $v\in Z$ then $y\force v\in\dot{Z}$ and if $v\notin Z$ then $y\force v\notin \dot{Z}$.
We often suppress $\dot{Z}$ if it is clear from context, or irrelevant.
\end{defn}

\begin{remark}
If $y\sqsubseteq x$ defines $Z$ via $\dot{Z}$, and $x'$ is a $\mathbb{P}$-generic extending $y$, then $\dot{Z}[x']=Z$.
\end{remark}
From Lemma~\ref{lemma: restriction to dominating function} we conclude:
\begin{cor}\label{cor;subsets-of-ordinals-in-M}
Suppose $Z\subset\kappa$ is in $M$. Let $f\in\omega^\omega$ be such that
\begin{itemize}
    \item $f$ dominates $\set{f^{\check{\alpha},\dot{Z}}}{\alpha<\kappa}$;
    \item $f$ is pointwise above $d^{\dot{Z}}$.
\end{itemize}
Then $x\restriction f$ defines $Z$.
\end{cor}
\begin{cor}\label{cor: reals in M}
Suppose $Z\subset\omega$ in $M$ is a real. Then there is $f\in M\cap \omega^\omega$ such that $x\restriction f$ defines $Z$.
\end{cor}
\begin{proof}
Using DC, we may find in $M$ a sequence $f^{\check{n},\dot{Z}}$, each as in Lemma~\ref{lemma: restriction to dominating function}. Let $f\in M$ diagonalize this sequence, with $f\geq f^{\dot{Z}}$.
\end{proof}
So reals in $M$ are determined by $x\restriction f$ for some $f\in M$.
Therefore by changing $x$ above $f$, the realization of a particular real in $M$ in unchanged. We will want to change $x$ on all columns, to another $\P$-generic.
\begin{lem}\label{lem;alteration}
Suppose $x\in 2^{\omega\times\omega}$ is $\P$-generic over $V$, $f\colon\omega\to\omega$ is in $V[x]$.
Then there is some $x'\in 2^{\omega\times\omega}$ such that
\begin{itemize}
    \item $x'$ is $\P$-generic over $V$;
    \item $x$ and $x'$ agree below $f$;
    \item $x \not\mathrel{E_1} x'$.
\end{itemize}
Moreover, both $x,x'$ live in a Cohen-real extension of $V$.
\end{lem}

\begin{proof}\footnote{Another way of finding $x'$ as in the lemma, due to Paul Larson, is as follows. Let $y\in 2^{\omega\times\omega}$ be $\P$-generic over $V[x]$, and define $x'$ to agree with $x$ below the graph of $f$, and to agree with $y$ otherwise. It can be shown that $x'$ is $\P$-generic over $V$.}
Let $\dot{f}$ be a name such that $\dot{f}[x]=f$.
Define a poset $\mathbb{Q}$ as follows.
The conditions are pairs $(p,q)$ where $p\in\mathbb{P}$, $q\colon \dom q\to\omega$ is a finite function such that $p\force q(l)\geq \dot{f}(l)$ for any $l\in\dom q$.
Say that $(p,q)$ extends $(p',q')$ if $p\leq p'$, $q$ extends $q'$ as a function, and for any $i\in\dom q\setminus \dom q'$, $(i,q(i))$ is above the domain of $p'$. (That is, if $(i,j)\in\dom p'$ then $j<q(i)$.)
Note that $\Q$ is countable and therefore forcing equivalent to Cohen-real forcing.

Define $\pi\colon\mathbb{Q}\to\mathbb{P}$ as follows. $\pi(p,q)=r$ where $\dom r=\dom p$ and
\begin{equation*}
    r(i,j) = \begin{cases}
               1-p(i,j)               & i\in\dom q \wedge j=q(i);\\
               p(i,j)               & \text{otherwise}.
           \end{cases}
\end{equation*}
That is, $r$ flips the values of $p$ along the graph of $q$, and is equal to $p$ elsewhere.
It follows from the definitions that $\pi$ is a forcing projection (see \cite[5.2]{Cummings2010-handbook}). 
%That is, $\pi$ respects the partial orders, and for any $(p_1,q_1)\in\Q$, if $r\leq_\P \pi(p_1,q_1)$, there is $(p_2,q_2)\leq_\Q (p_1,q_1)$ so that $\pi(p_2,q_2)\leq_\P r$.

A $\mathbb{Q}$-generic is naturally associated with a pair $(y,g)$ such that $y\in 2^{\omega\times\omega}$ is $\mathbb{P}$-generic and $g\colon\omega\to\omega$ is a function which is pointwise above $\dot{f}[y]$.
Note also that the map $\pi_0\colon \mathbb{Q}\to\mathbb{P}$ defined by $(p,q)\mapsto p$ is a projection, so we may find $g$ so that $(x,g)$ is $\mathbb{Q}$-generic over $V$.
Let $x'=\pi(x,g)$, a $\mathbb{P}$-generic over $V$.
Then $x$ and $x'$ agree everywhere but on the graph of $g$. Since $g$ is pointwise above $f=\dot{f}[x]$, $x$ and $x'$ satisfy the desired properties.
\end{proof}

We now prove part (2) of Theorem~\ref{thm;E_1-classification}.
\begin{prop}
There is no generic classification of $E_1$ using sequences of $E_0$-classes of length $<\add$. 
\end{prop}
Let $E=E_0$ for this proof. Note that the proof works without change for any countable Borel equivalence relation $E$.
\begin{proof}
Assume for a contradiction that there is such a classification $c\colon (2^\omega)^\omega\to I$, defined by some $\psi(x,B,e)$, where $e$ is a parameter and $I$ is the set of sequences of $E$-classes of length $<\add$.
Let $x\in (2^\omega)^\omega$ be Cohen generic and $c(x)=\seqq{B_\alpha}{\alpha<\kappa}=B$ its invariant, where $\kappa<\mathfrak{b}$ and each $B_\alpha$ is an $E$-class. Then $B$ is in $M$. Indeed, let $y_n$ be defined as $x$ on the $n$'th column forward, and be all $0$'s on the first $n$ columns. Then $y_n\mathrel{E_1}x$ and $y_n\in V[x_n]$. Now $c(x)=c(y_n)\in V[x_n]$ for each $n$, so $B$ is in the intersection model $M$.

Working in $V[x]$, fix $\seqq{b_\alpha}{\alpha<\kappa}$ such that $b_\alpha\in B_\alpha$. Each $b_\alpha$ is a member of a Polish space, and therefore can be identified as a subset of $\omega$.
Since $b_\alpha\in M$, by Corollary~\ref{cor: reals in M}, there is a function $f_\alpha$ such that for any $h$, if $h$ is above $f_\alpha$ then $x\restriction h$ defines $b_\alpha$.

%Since $b_\alpha\in M$ there are functions $f_\alpha$ and $d_\alpha$ such that for any $h$, if $h$ dominates $f_\alpha$ and $h$ is above $d_\alpha$, then $x\restriction h$ defines $b_\alpha$.

By a result of Cichon and Pawlikowsky \cite{Chichon-Pawlikowski-1986}, the bounding number $\mathfrak{b}$ in a Cohen-real extension $V[x]$ is $\add^V$, the additivity of the meager ideal as calculated in $V$.
Since $\kappa<\add^V$, $\kappa<\mathfrak{b}^{V[x]}$, so in $V[x]$ we may find $g\in\omega^\omega$ which dominates $\set{f_\alpha}{\alpha<\kappa}$.
Fix a condition $p\in\P$ forcing that
$$\psi(\dot{x},\dot{B},\check{e})\wedge \forall\alpha<\kappa (B_\alpha=[\dot{b}_\alpha]_{E}).$$

By Lemma~\ref{lem;alteration} we may find $x'$, in a further generic extension, such that $x'$ is $\mathbb{P}$-generic over $V$, $x'$ and $x$ agree below $g$, yet $x$ and $x'$ are not $E_1$-equivalent. By a finite change of $x'$, we may assume that $x'$ extends $p$.

\begin{claim}
For each $\alpha<\kappa$, $\dot{B}_\alpha[x']=\dot{B}_\alpha[x]=B_\alpha$.
\end{claim}
\begin{proof}
Fix $\alpha<\kappa$. Define $y\in 2^{\omega\times\omega}$ by changing finitely many coordinates of $x'$, so that $y$ still extends $p$ and agrees with $x$ below $g$, and moreover $y$ and $x$ agree below $f_\alpha$. (This is possible since $g$ eventually dominates $f_\alpha$.)
Then $y$ is $\P$-generic over $V$ and $\dot{b}_\alpha[y]=\dot{b}_\alpha[x]=b_\alpha$, by the choice of $f_\alpha$.
Furthermore, $y$ and $x'$ are $E_1$-related. Since the map $c$ is a generic classification, $\dot{B}[y]=\dot{B}[x']$. In particular, $\dot{B}_\alpha[x']=\dot{B}_\alpha[y]=[b_\alpha]_E=B_\alpha$.
\end{proof}
It follows that $c(x)=B=c(x')$.
In conclusion, we found $x$ and $x'$ in a Cohen-real generic extension which are not $E_1$-equivalent, yet they are assigned the same invariant by $c$. Thus $c$ is not a complete classification in this extension, so $c$ is not a generically absolute classification.
\end{proof}

The following establishes part (3) of Theorem~\ref{thm;intersection-model}.
Recall that, working in some generic extension of $V$, given a set $A$, $V(A)$ is the minimal transitive extension of $V$ which satisfies ZF. Showing that a model $M$ is of the form $V(A)$ for a set $A$, where $V$ is a model of ZFC, can be seen as some weak fragment of choice for $M$. If $A$ is a set of ordinals, then $V(A)$ satisfies the axiom of choice. The next level of complexity, in which choice can fail, is when $A$ is a set of reals (a set of subsets of $\omega$).
\begin{thm}
Let $\mathcal{F}=(\omega^\omega)^M$ and $D=\set{[x\restriction f]}{f\in\mathcal{F}}$.
%(any other dominating family in $M$ would work as well).
Then
\begin{equation*}
    M=V(D).
\end{equation*}
\end{thm}
Note that $V(D)=V(\bigcup D)$ where $\bigcup D$ is a set of reals.
\begin{proof}
As before, since the map $x\mapsto [x\restriction f]$ is $E_1$-invariant for each $f$, we see that $D\in M$, and therefore $V(D)\subset M$.
It remains to show that $M\subset V(D)$.
Note that each $f\in\mathcal{F}$ is encoded by $[x\restriction f]$, as the domain of all members in $[x\restriction f]$, and so $\mathcal{F}\in V(D)$ as well.
\begin{claim}\label{claim;forcing-over-V(D)}
There is a poset $\mathbb{Q}\in V(D)$ and a filter $G\subset \mathbb{Q}$ which is $\Q$-generic over $V[x]$, such that in $V(D)[G]$ there a $y\in (2^\omega)^\omega$, Cohen generic over $V$, and there is a function $g\in V[x][G]$ dominating all functions in $\mathcal{F}$ so that $y\restriction g = x\restriction g$.
\end{claim}

\begin{proof}
Consider the poset $\mathbb{Q}\in V(D)$ to enumerate $\mathcal{F}$, $D$, and $(\mathcal{P}(\mathbb{P}))^V$. Fix $G\subset\mathbb{Q}$ generic over $V[x]$. In the extension $V(D)[G]$
let $\seqq{E_n}{n<\omega}$ be an enumeration of the dense open subsets of $\mathbb{P}$ in $V$.
Next, find in $V(D)[G]$ sequences $\seqq{y_n}{n<\omega}$ and $\seqq{h_n}{n<\omega}$ such that 
\begin{itemize}
    \item $h_n\in \mathcal{F}$, $y_n\in [x\restriction h_n]$;
    \item $\seqq{h_n}{n<\omega}$  is $<$-increasing and $<^\star$-dominating $\mathcal{F}$;
    \item $y_{n+1}$ extends $y_n$;
    \item $y_n$ extends a condition in $E_n$.
\end{itemize}
This can be done as follows. Take $h_{n+1}\in\mathcal{F}$ which is strictly above $h_n$, such that $h_{n+1}$ dominates the $n+1$'th member of $\mathcal{F}$ (under some fixed enumeration in $V(D)[G]$), and such that there exists some member of $[x\restriction h_{n+1}]$ extending a condition in $E_{n+1}$.
To find $y_{n+1}$, take any member of $[x\restriction h_{n+1}]$ (the minimal according to the enumeration given by the collapse), and change it in finitely many positions so that it extends $y_n$.

Define $y=\bigcup_n y_n$, so $y\in 2^{\omega\times\omega}$ is $\P$-generic over $V$.
Work now in $V[x][G]$. For each $n$, there is some $m_n$ such that $y_n$ and $x\restriction h_n$ agree past $m_n$. We may assume that $m_{n}<m_{n+1}$. Define $g\in\omega^\omega$ by $g(k)=h_n(k)$ if $n$ such that $m_n<k\leq m_{n+1}$, and $g(k)=0$ if $k\leq m_0$. 
Then $x\restriction g=y\restriction g$. 
Given $f\in\mathcal{F}$, there are $n$ and $m$ such that for $k>m$, $f(k)<h_n(k)<g(k)$.
We therefore established that $g$ dominates each $f\in\mathcal{F}$, which concludes the proof of the claim.
\end{proof}

\begin{lemma}\label{lem;dominating-defines-M}
Suppose $y\in 2^{\omega\times\omega}$ is $\P$-generic over $V$, in some further generic extension, and $y$ extends $x\restriction g$ for some $g$ which dominates $\omega^\omega\cap M$.
Them $M\subseteq V[y]$.
\end{lemma}

\begin{proof}
Note that for any $Z\in M$ we can find some $y'$, a finite alteration of $y$, so that $y'$ and $x$ agree above $d^{\dot{Z}}$ and above $g$.
This is because $g$ dominates $d^{\dot{Z}}$, so $y$ and $x$ agree eventually above $d^{\dot{Z}}$.
\begin{claim}
For any $Z$ in $M$, if $y'$ is a finite alteration of $y$ such that $y'$ agrees with $x$ above $d^{\dot{Z}}$ and $g$, then $\dot{Z}[y']=Z$.
\end{claim}
In this case, $V[y']=V[y]$, so we conclude that $Z\in V[y]$ for any $Z\in M$, as the lemma requires.
We now prove the claim, by induction on the rank of the set $Z$.

Fix $Z$ in $M$, and let $y'$ be as above.
Fix some $z\in M$ of lower rank, and assume that $z\in Z$. We need to show that $z\in \dot{Z}[y']$. (The argument for when $z\notin Z$ is similar.)
Let $\dot{z}$ be a name such that $\dot{z}[x]=z$, and $y''$ a finite alteration of $y'$ which agrees with $x$ also above $d^{\dot{z}}$ as well as $d^{\dot{Z}}$.

Since $y'$ and $y''$ differ by only finitely many coordinates, and both force that $\dot{Z}=\tau^{\dot{Z}}_n$ for all $n$, then $\dot{Z}[y']=\dot{Z}[y'']$.
$y''$ agrees with $x$ above $g$ and $d^{\dot{z}}$, so $\dot{z}[y'']=z$ by the inductive hypothesis. 
Furthermore, since $g$ dominates
$d^{\dot{z},\dot{Z}}$, then $\dot{z}[y'']\in \dot{Z}[y'']$ by Lemma~\ref{lemma: restriction to dominating function}. It follows that $z\in\dot{Z}[y']$, as desired.
\end{proof}

\begin{lem}[Folklore]\label{lem;mutgen}
Suppose $N\subset M$ are models of ZF, $\mathbb{P}\in N$ is a poset. If $G$ is $\mathbb{P}$-generic over $M$, then $N[G]\cap M=N$.
\end{lem}
Finally, take $G$ and $y\in V(D)[G]$ as above.
Since $G$ is generic over $V[x]$, it is generic over $M$, and therefore $V(D)[G]\cap M\subset V(D)$ by Lemma~\ref{lem;mutgen}.
Since $M\subset V[y]\subset V(D)[G]$, it follows that $M\subset V(D)$, concluding the proof of the theorem.
\end{proof}

We now prove part (3) of our main Theorem~\ref{thm;E_1-classification}, that there is no complete classification for $E_1$ which is absolute for all forcing.
\begin{proof}[Proof of part Theorem~\ref{thm;E_1-classification} part (3)]
Assume towards a contradiction that $x\mapsto A_x$ is an absolute classification of $E_1$. Let $A=A_x$, the classifying invariant of the Cohen generic $x\in (2^\omega)^\omega$. Then $A\in M=V(D)$.
By \cite[Theorem B]{Grigorieff-1975}
there is some poset $\mathbb{Q}$ in $V(D)$ and a $\mathbb{Q}$-generic filter over $V(D)$, $H$, such that $V(D)[H]=V[x]$.
Fix a $\mathbb{Q}$-name $\sigma$ such that $\sigma[H]=x$ and a condition $q\in\mathbb{Q}$ forcing, in $V(D)$, that $A_{\sigma}=\check{A}$.

Let $q\in H'$ be $\mathbb{Q}$-generic over $V(D)[H]$, and let $x'=\sigma[H']$, so $A_x=A=A_{x'}$. We claim that $x$ and $x'$ are not $E_1$-related, showing that $x\mapsto A_x$ is not a complete classification in this extension.
Indeed, by Lemma~\ref{lem;mutgen}, $V(D)[H]\cap V(D)[H']=V(D)=M$. If $x$ were $E_1$-related to $x'$, then for large enough $n$, they would agree on their n'th columns. This would imply that for large $n$, the n'th column $x(n)$ is in $V[x]$ and $V[x']$, and therefore is in $M$. This is a contradiction, since $x(n)\not\in V[x_{n+1}]$ and $M\subset V[x_{n+1}]$.

\end{proof}

\section{Turbulence}\label{section: turbulence}
We now show that turbulent equivalence relations remain completely unclassifiable in the context of generically absolute classifications.
When $E=E_a$ is the orbit equivalence relation of a turbulent action $a$, this follows directly from the characterization of turbulence in \cite{Larson_Zapletal_2020} (Theorem~\ref{thm: turbulence characterization} below).
In this section we extend the characterization of Larson and Zapletal (Theorem~\ref{thm: turbulence characterization intersection model}), and use this to extend the definition of turbulence to arbitrary analytic equivalence relations (Definition~\ref{defn: generically turbulent double brackets}).
Finally, we prove in this generality that turbulent equivalence relations are not generically classifiable (Proposition~\ref{thm: generic turb not generic classifiable}).

\begin{thm}[Larson-Zapletal {\cite[Theorem 3.2.2]{Larson_Zapletal_2020}}]\label{thm: turbulence characterization}
Suppose $a\colon G\curvearrowright X$ is a continuous action of a Polish group $G$ on $X$ with dense and meager orbits.
The following are equivalent.
\begin{itemize}
    \item $a\colon G\curvearrowright X$ is generically turbulent;
    \item If $x\in X$ is Cohen-generic over $V$ and $g\in G$ is Cohen generic over $V[x]$ then $V[x]\cap V[g\cdot x]=V$. 
\end{itemize}
\end{thm}

\begin{defn}[{\cite[Definition 3.10]{Kanovei-Sabok-Zapletal-2013}}]\label{def;double-brackets}
Let $E$ be an analytic equivalence relation on a Polish space $X$, and let $x\in X$ be in some generic extension of $V$.
Define
\begin{equation*}
    V[[x]]_E=\bigcap\set{V[y]}{y\textrm{ is in some further generic extension, }y\in X\textrm{ and }x\mathrel{E}y}.
\end{equation*}
That is, a set $b$ is in $V[[x]]_E$ if in any generic extension of $V[x]$ and any $y$ in that extension which is $E$-equivalent to $x$, $b$ is in $V[y]$.
\end{defn}
We will call this model $V[[x]]_E$ the \textbf{intersection model} of $x$, with respect to $E$.
For example, our tail intersection model $M$ is precisely $V[[x]]_{E_1}$.
Indeed, it follows from the definition that $V[[x]]_{E_1}\subset V[x_n]$ for each $n$, and therefore $V[[x]]_{E_1} \subset M$. On the other hand, for any $y\mathrel{E_1}x$, in some generic extension, there is some $n$ such that $V[x_n]\subset V[y]$. Thus $M\subset V[y]$ for any such $y$, and so $M\subset V[[x]]_{E_1}$.

%Kanovei, Sabok, and Zapletal~\cite{Kanovei-Sabok-Zapletal-2013} studied canonization properties of equivalence relations with respect to various ideals on their domain. In \cite{Sha18} the intersection model was used to study Borel reducibility, particularly for equivalence relations which are classifiable by countable structures.

In the setting of Theorem~\ref{thm: turbulence characterization}, note that $V[[x]]_{E_a}\subset V[x]\cap V[g\cdot x]$, for any $g$, where $E_a$ is the orbit equivalence relation induced by the action $a$.
It follows from Theorem~\ref{thm: turbulence characterization} that if $a\colon G\curvearrowright X$ is a generically turbulent action, then for a Cohen generic $x\in X$, $V[[x]]_{E_a}=V$ is as small as can be.
We show that the latter already characterizes turbulence.
\begin{thm}\label{thm: turbulence characterization intersection model}
Suppose $a\colon G\curvearrowright X$ is a continuous action of a Polish group $G$ on $X$ with dense and meager orbits.
The following are equivalent.
\begin{itemize}
    \item $a\colon G\curvearrowright X$ is generically turbulent.
    \item If $x\in X$ is Cohen-generic over $V$ then $V[[x]]_{E_a}=V$.
\end{itemize}
\end{thm}

The key lemma in the proof of Theorem~\ref{thm: turbulence characterization intersection model} is the following. 
\begin{lem}\label{lem;actionable-intersection-of-two}
Suppose $a\colon G\curvearrowright X$ is a continuous action of a Polish group $G$ on $X$ and $E_a$ is the induced equivalence relation.
Let $x\in X$ be in \textit{some} generic extension and $g\in G$ be Cohen generic over $V[x]$. Then
\begin{equation*}
    V[[x]]_{E_a}=V[x]\cap V[g\cdot x].
\end{equation*}
\end{lem}
Note that here $x$ is not required to be Cohen generic.
\begin{proof}
Let $z=g\cdot x$.
By definition, $V[[x]]_{E_a}\subset V[x]\cap V[z]$.
It remains to show that for any $y$ in a generic extension of $V[x]$, if $y\mathrel{E}x$ then $V[x]\cap V[z]\subset V[y]$.
Suppose first that $y\in V[x][H]$ where $H$ is $\mathbb{P}$-generic over $V[x][g]$ (not just over $V[x]$) for some $\mathbb{P}\in V[x]$. In this case, by mutual genericity, $g$ is generic over $V[x][H]$.
In $V[x][H]$, fix $\gamma\in G$ such that $y=\gamma\cdot x$.
Since $G$ acts on itself by homeomorphisms and $g\in G$ is Cohen over $V[x][H]$, then so is $g\gamma$, and therefore $V[y][g\gamma]\cap V[x][H]=V[y]$, by Lemma~\ref{lem;mutgen}.
Finally, $g\gamma\cdot y=g\cdot x=z$ is in $V[y][g\gamma]$, so $V[z]\cap V[x]\subset V[y][g\gamma]\cap V[x][H]=V[y]$,
as desired.

For the general case, let $y\in V[x][H]$ where $H$ is some $\mathbb{P}$-generic over $V[x]$, $\mathbb{P}\in V[x]$.
$H$ may not be generic over $V[x][g]$.
It suffice to show that if $a\in V[x]$ and $a\notin V[y]$ then $a\notin V[z]$.
Fix an $a\in V[x]$ and some condition $p$ forcing that $x\mathrel{E}\dot{y}$ and $\check{a}\notin V[\dot{y}]$.
Let $H'$ be $\mathbb{P}$-generic over $V[x][g]$ extending $p$. By the argument above $V[z]\cap V[x]\subset V[\dot{y}[H']]$.
Since $a\notin V[\dot{y}[H'
]]$ then $a\notin V[z]$.
\end{proof}

%Theorem~\ref{thm: turbulence characterization intersection model} is a direct consequence of Theorem~\ref{thm: turbulence characterization} and the above lemma.
\begin{proof}[Proof of Theorem~\ref{thm: turbulence characterization intersection model}]
For the forward direction: take a Cohen generic $g\in G$ over $V[x]$, then $V[[x]]_{E_a}\subset V[x]\cap V[g\cdot x]=V$ (by Theorem~\ref{thm: turbulence characterization}).
For the other direction, assume that for any Cohen generic $x\in X$ the intersection model is trivial: $V[[x]]_{E_a}=V$. Then for any $g\in G$, Cohen generic over $V[x]$, it follows from Lemma~\ref{lem;actionable-intersection-of-two} that $V[x]\cap V[g\cdot x]=V[[x]]_E=V$. Now Theorem~\ref{thm: turbulence characterization} implies that $a\colon G\curvearrowright X$ is generically turbulent.
\end{proof}

These results naturally motivate the following definition of turbulence.

\begin{defn}\label{defn: generically turbulent double brackets}
Let $E$ be an analytic equivalence relation on a Polish space $X$.
Say that $E$ is \textbf{generically turbulent} if for any Cohen generic $x\in X$,
\begin{enumerate}
    \item the $E$-class of $x$ is dense and meager in $X$ and
    %\item each $E$-class is meager and dense in $X$;
    \item $V[[x]]_E=V$.
\end{enumerate}
\end{defn}
\begin{remark}
If $a\colon G\curvearrowright X$ is a continuous action and $E$ is the orbit equivalence relation, it follows from Theorem~\ref{thm: turbulence characterization intersection model} that the action is generically turbulent if and only if $E$ is generically turbulent with respect to the above definition.
\end{remark}

%Following their Theorem~\ref{thm: turbulence characterization}, Larson and Zapletal \cite[Theorem 3.3.5]{Larson_Zapletal_2020} gave a forcing based proof for the fact that generically turbulent equivalence relations are not classifiable by countable structures (Theorem 12.5 in \cite{Kechris_1997-classification-problems}). 
%(In fact, they give a stronger result, where the class of equivalence relations which are classifiable by countable structures is replaced with a larger class, which they call virtually placid equivalence relations.)
%We follow these ideas to prove Proposition~\ref{prop: generic turb not generic classifiable}, in the context of Definition~\ref{defn: generically turbulent double brackets}.

%Recall the following well known ``mutual genericity'' lemma.\begin{lem}\label{lem;mutgen}Suppose $N\subset M$ are models of ZF, $P\in N$ is a poset. If $x$ is $P$-generic over $M$, then $N[x]\cap M=N$.\end{lem}

\begin{thm}\label{thm: generic turb not generic classifiable}
If $E$ is generically turbulent (as in Definition~\ref{defn: generically turbulent double brackets}) then $E$ does not admit a generically absolute classification.
\end{thm}
\begin{proof}
Assume for a contradiction that $x\mapsto A_x$ is a generically absolute classification of $E$.
Let $x\in X$ be a Cohen generic over $V$. Then $A=A_x$ is in $V[[x]]_E$. By assumption, $A\in V$.
Fix a condition $p$ in the Cohen forcing for $X$ such that $p\force A_{\dot{x}}=\check{A}$.
Now any two Cohen generics extending $p$ are $E$-related, contradicting the fact that the $E$-classes are meager.
\end{proof}

Since $E_1$ admits a generically absolute classification (Claim~\ref{claim: complete classification of E1}), then so does $E_1^+$ (see Proposition~\ref{prop: jump is classifiable}).
This implies the result of Kanovei and Reeken \cite{Kanovei-Reeken}, that no turbulent orbit equivalence relation is Borel reducible to $E_1^+$.
This result is vastly generalized in \cite[Theorem 3.3.5]{Larson_Zapletal_2020}.

\subsection{Intersection numbers and a theorem of Kechris and Louveau}\label{subsection: intro intersection numbers}

\begin{thm}[{Kechris-Louveau \cite[Theorem 4.2]{Kechris-Louveau-1997}}]\label{thm;e1-not-action}
Suppose $a\colon G\curvearrowright X$ is a continuous action of a Polish group $G$ on a Polish space $X$, let $E_a$ be the induced orbit equivalence relation on $X$.
Then, on any comeager subset of $(2^\omega)^\omega$, $E_1$ is not Borel reducible to $E_a$.
\end{thm}
Larson and Zapletal \cite[Theorem 4.1.1]{Larson_Zapletal_2020} gave a proof of this theorem using the intersection model $M$. We give a different proof here, using the following definition, motivated by Lemma~\ref{lem;actionable-intersection-of-two}.

\begin{defn}
Given an equivalence relation $E$ on $X$ and $x\in X$ in some generic extension, let \textbf{the intersection number of $x$ (relative to $E$)} be the minimal size of a finite set $B$ such that 
\begin{equation*}
    V[[x]]_E=\bigcap_{y\in B} V[y],
\end{equation*}
where $B$ is contained in the $E$-class of $x$ in some generic extension of $V[x]$.
If no such set exists say that the intersection number is infinite.
\end{defn}

It follows from Lemma~\ref{lem;actionable-intersection-of-two} that if $E$ is an orbit equivalence relation then the intersection number of $x$ is always $\leq 2$, for any $x\in X$ in some generic extension.
On the other hand, for a Cohen generic $x\in (2^\omega)^\omega$, the intersection number of $x$ relative to $E_1$ is infinite:
\begin{claim}\label{claim;E1-inf-intersection}
Let $x\in (2^\omega)^\omega$ be Cohen-generic.
Suppose $x_1,...,x_n$, in some further generic extension, are all $E_1$-equivalent to $x$.
Then 
\begin{equation*}
 V[[x]]_{E_1}\subsetneq V[x_1]\cap ...\cap V[x_n].   
\end{equation*}
\end{claim}
\begin{proof}
Fix $k$ large enough such that $x$ and $x_1,...,x_n$ all agree starting the $k$'th column onwards.
Then $x(k)\in  V[x_1]\cap ...\cap V[x_n]$.
However, $x(k)\notin V[[x]]_{E_1}=M$, so $M$ is strictly smaller than $V[x_1]\cap ... \cap V[x_n]$.
\end{proof}

Borel reductions respect the intersection number:

\begin{lem}\label{lem;reduction-intersection-number}
Suppose $f\colon E\to_B F$ is a (partial) Borel reduction and $x\in\dom f$ in some generic extension. 
Then the intersection number of $x$ relative to $E$ is equal to the intersection number of $f(x)$ relative to $F$.
\end{lem}
By a partial Borel reduction we mean that the domain of $f$ is a Borel subset of the domain of $E$, and $f$ is a reduction of $E\restriction \dom f$ to $F$.
The lemma is a generalization of \cite[Lemma 3.5]{Sha18}, which states that if $f\colon X\to Y$ is a (partial) Borel reduction of $E$ to $F$, and $x\in\dom f$, in some generic extension, then  $V[[x]]_E=V[[f(x)]]_F$.
\begin{proof}
%By Lemma~\ref{lem;reduction-doublebrackets}, $V[[x]]_E=V[[f(x)]]_F$.
Assume first that $V[[f(x)]]_F=\bigcap_{y\in B}V[y]$ where $B$ is contained in the $F$-class of $f(x)$ in some big generic extension $V[G]$.
For each $y\in B$, $f(x) \mathrel{F} y$ in $V[G]$. By absoluteness for the statement $(\exists x)f(x) \mathrel{F} y$ there is $x_y\in V[y]$ such that $f(x_y)\mathrel{F}y$, thus $x_y \mathrel{E} x$ for each $y\in B$ (since $f$ is a reduction). 
Now $\bigcap_{y\in B}V[x_y]\subset \bigcap_{y\in B}V[y]=V[[f(x)]]_F=V[[x]]_E$. It follows that $\bigcap_{y\in B}V[x_y]=V[[x]]_E$, so the intersection number of $x$ is $\leq |B|$.

Similarly, suppose $V[[x]]_E=\bigcap_{y\in D}V[y]$ where $D$ is contained in the $E$-class of $x$, in some big generic extension. Then $\set{f(y)}{y\in D}$ is contained in the $F$-class of $f(x)$, and $V[[f(x)]]_F=\bigcap_{y\in B}V[f(y)]$, so the intersection number of $f(x)$ is $\leq|D|$. 
We conclude that the intersection numbers of $x$ and $f(x)$ are the same.
\end{proof}

\begin{proof}[Proof of Theorem~\ref{thm;e1-not-action}]
Assume for contradiction that there is a reduction $f$, defined on a comeager subset of $\R^\omega$, reducing $E_1$ to some equivalence relation $E$ induced by a Polish group action.
Let $x\in \R^\omega$ be Cohen generic over $V$, so $x$ is in the domain of $f$.
By lemmas \ref{lem;actionable-intersection-of-two} and \ref{lem;reduction-intersection-number} it follows that the intersection number of $x$ is $2$, contradicting claim~\ref{claim;E1-inf-intersection}.
\end{proof}
\begin{remark}
In the proof above, $x$ is taken to be a Cohen generic real in the domain of $E_1$. However, we have no control over what $f(x)$ looks like. This is why we need to consider arbitrary reals $x\in X$ in Lemma~\ref{lem;actionable-intersection-of-two}.
\end{remark}

\begin{question}[see \cite{Kechris-Louveau-1997}]\label{question;KL}
If $E$ is an analytic equivalence relation, is it true that either $E$ is Borel reducible to an orbit equivalence relation or $E_1\leq E$?
\end{question}

The proof above suggests the following strategy for a counterexample: suppose we can find an analytic equivalence relation $E$ such that:
\begin{enumerate}
    \item For any $x\in \dom E$ in a generic extension the intersection number of $x$ is finite;
    \item there is $x\in \dom E$ in a generic extension whose intersection number is strictly greater than 2.
\end{enumerate}
Part (1) would imply that $E_1\not\leq_B E$ and part (2) that $E$ is not Borel reducible to an orbit equivalence relation.
On the other hand, the following would support a positive answer to Question~\ref{question;KL}. 
\begin{question}
If $E$ is an analytic equivalence relation, $x$ in the domain of $E$ in some generic extension, must the intersection number of $x$ be either $\leq 2$ or infinite?
\end{question}

\section{Chromatic numbers in the intersection model}\label{sec;chromatic-number}

Consider the shift graph $(\mathcal{P}(\omega),S)$, where for $X\subset \omega$, $S(X)=X\setminus\min X$. (We identify $\mathcal{P}(\omega)$ with $2^\omega$ in the usual way.)
This is an acyclic graph, so its chromatic number is 2 in a ZFC model.
\begin{prop}\label{prop;ch-number-not-2}
In $M$, the chromatic number of the shift graph is not 2.
\end{prop}

\begin{remark}
The following alternative intersection model was studied in \cite{Larson_Zapletal_2020}. Let $c\colon\omega\times\omega_1\to\{0,1\}$ be Cohen-generic over $V$. Define $M'=\bigcap_{n<\omega}V[c\restriction(\omega\setminus n)\times\omega_1]$.
This model was considered precisely to provide an example of an intersection model in which the axiom of choice fails. Indeed, they show that in $M'$ the shift graph does not have chromatic number 2.
The argument below is similar, using function in $\omega^\omega$ to get uncountably many reals out of $\omega$ many Cohen reals, as before.
Note that both $M'$ and $M$ are models of DC.
\end{remark}
%\begin{remark}The chromatic number of the shift graph is 2 if and only if the chromatic number of $G_0$ is 2.\end{remark}
\begin{proof}[Proof of Proposition~\ref{prop;ch-number-not-2}]
In $V$, let $\seqq{f_\alpha}{\alpha<\mathfrak{b}}$ be an $<^\ast$-increasing and unbounded sequence of function in $\omega^\omega$ such that each $f_\alpha$ is increasing.
Let $x\in 2^{\omega\times\omega}$ be the Cohen generic over $V$, where $M=\bigcap_{n<\omega}V[x_n]$, $x_n=x\restriction (\omega\setminus n)\times\omega$.
Define $z_\alpha\in 2^\omega$ by $z_\alpha(n)=x(n,f_\alpha(n))$, the restriction of $x$ to the graph of $f_\alpha$.

As before, each $z_\alpha$ is in $M$, but $\seqq{z_\alpha}{\alpha<\mathfrak{b}}$ is not in $M$ (similar to Lemma~\ref{lem;no-choice-b-E0-classes}).
For $m<\omega$, define $z_{\alpha,m}\in 2^\omega$ by $z_{\alpha,m}(k)=
    \begin{cases}
               0 & k < m;\\
               z_\alpha(k) & k \geq m.
    \end{cases}
$ 
So $z_{\alpha,0}=z_\alpha$ and $z_{\alpha,m}$ is the result of finitely many application of $S$ to $z_\alpha$.
Let $A=[\set{z_\alpha}{\alpha<\mathfrak{b}}]_{E_0}$, all reals for which there is some $z_\alpha$ with which they agree up to finitely many values.
Note that $A\in M$ and $\set{z_{\alpha,m}}{\alpha<\mathfrak{b},\,m<\omega}\subset A$.
We will show that in $M$ there is no 2-coloring of $(A,S)$.

Assume towards a contradiction that $\sigma\colon A\to \{0,1\}$ is a coloring, $\sigma\in M$. 
Let $\dot{\sigma}$ be a name forced to be a 2-coloring of $A$ in $M$.
Working in $V$, find conditions $p_\alpha\in\mathbb{P}$ and $\epsilon_\alpha\in\{0,1\}$ such that $p_\alpha\force\dot{\sigma}(\dot{z}_\alpha)=\epsilon_\alpha$ for any $\alpha\in X$.
Since $\mathbb{P}$ is countable, we may find an unbounded $X\subset\mathfrak{b}$, a single condition $p\in\mathbb{P}$ and $\epsilon\in\{0,1\}$ such that for $\alpha\in X$, $p_\alpha=p$ and $\epsilon_\alpha=\epsilon$.
Since $\seqq{f_\alpha}{\alpha\in X}$ is $<^\ast$-unbounded there is $m<\omega$ for which $\set{f_\alpha(m)}{\alpha\in X}$ is unbounded in $\omega$.

Let $\dot{\sigma}_{m+1}$ be a $\mathbb{P}_{m+1}$-name and $q\in\mathbb{P}$ extending $p$ such that $q\force \dot{\sigma}[\dot{x}]=\dot{\sigma}_{m+1}[\dot{x}_{m+1}]$. (This is possible since $\sigma$ is forced to be in $V[\dot{x}_{m+1}]$.)
Fix $\alpha\in X$ such that $(m,f_\alpha(m))$ is not in the domain of $q$.
Working with the poset $\mathbb{P}_{m+1}$, consider the condition $q\restriction (\omega\setminus (m+1)\times\omega)$ and extend it to $t\in\mathbb{P}_{m+1}$ such that $t\force \dot{\sigma}_{m+1}(\dot{z}_{\alpha,m+1})=\delta$ for some $\delta\in\{0,1\}$.
Now $q\cup t$ is a condition in $\mathbb{P}$ forcing that $\dot{\sigma}(\dot{z}_\alpha)=\epsilon$ and $\dot{\sigma}(\dot{z}_{\alpha,m+1})=\delta$.

The only coordinates on which $z_\alpha$ and $z_{\alpha,m+1}$ differ are $0,...,m$, and at least one of these, the $m$'th coordinate, is undecided by $q\cup t$. So we can find two extensions of $q\cup t$, $r_0,r_1$ such that for each $i=0,1$, $r_i$ decides $z_\alpha(0),...,z_\alpha(m)$ and that the parity of $|\set{k\in\{0,...,m\}}{z_\alpha(k)=1}|$ is $i$, respectively. 
Since $\sigma$ is a 2-coloring, $r_0$ forces that $z_\alpha$ and $z_{\alpha,m+1}$ have the same color, that is, $\epsilon=\delta$, and $r_i$ forces that $\epsilon\neq\delta$, a contradiction.
\end{proof}

\begin{question}
What is the chromatic number of $(\mathcal{P}(\omega),S)$ in $M$?
\end{question}

\begin{remark}
When restricting to just the set $A$ defined above, the graph $(A,S)$ has chromatic number precisely 3 in $M$.
Conley and Miller \cite{Conley-Miller-16} proved in a very general setting that acyclic graphs, such as the shift graph, admit Borel 3-colorings, when restricted to a comeager set.
The set $A$ is ``sufficiently generic'' for their argument, as follows.
It suffices to find a forward recurring independent subset of $A$. 
Let $B=$ all $z\in A$ of the form $z=0...0110\ast\ast\ast...$. That is, all $z$'s in which the first appearance of 1 is followed by 10.
$B$ is forwards recurring, by genericity of the members of $A$, and $B\cap S(B)=\emptyset$, as any $y\in S(B)$ is of the form $y=0...010\ast\ast\ast$.
\end{remark}

\section{Further questions}

Let $\mathcal{P}$ be a collection of forcing notions. Say that a map $c$ is a $\boldsymbol{\mathcal{P}}$-\textbf{absolute classification} of $E$ if, as in Definition~\ref{defn;generic-classification}, it satisfies (1), (2), and (3)(a), and (3)(b) for all generic extensions by posets in $\mathcal{P}$.
This notion still respects Borel reducibility, as in Remark~\ref{remark: gen class respects reducibility}. It may not respect the usual non-classifiability results. For example, a turbulent equivalence relation may become classifiable.

Following \cite{Zapletal-idealized-2008} and \cite{Kanovei-Sabok-Zapletal-2013}, it seems reasonable to study $\{P_I\}$-absolute classifications, where $I$ is a proper ideal on a Polish space, and $P_I$ is the poset of Borel sets modulo $I$. A particularly natural notion is a \textbf{random absolute classification}, where $I$ is the ideal of measure zero sets. In this measure theoretic context, Larson and Zapletal developed an analogue for turbulence \cite[Definition 3.6.1]{Larson_Zapletal_2020}.
By \cite[Theorem 3.6.2]{Larson_Zapletal_2020}, this measure-theoretic turbulence implies that $V[[x]]_E=V$ when $x$ is random-generic over $V$. It follows that for such equivalence relations no random absolute classification exists.

Similarly, given a proper ideal $I$ on $X^\omega$, it would be interesting to understand the tail intersection model corresponding to a $P_I$-generic $x\in X^\omega$.
This also relates to the study of intersection models of coherent sequences as in \cite[Chapter 4]{Larson_Zapletal_2020}.

\section{Appendix}
In this section we prove some statements about generic classifications which were mentioned in the introduction, Remark~\ref{remark: gen class respects reducibility}, Example~\ref{example: gen class =R no ord inv}, and Example~\ref{example: gen class E0 no power of ordinal invs}. These follow from well known facts about forcing. 
We also provide some more context and justification to the fact that ``absolute generic classification'' extends the notion of ``classification by countable structures''.
Some basic facts are presented in the more general context of $\mathcal{P}$-absolute classifications.
Fix a collections of posets $\mathcal{P}$.
\begin{remark}
Suppose $\mathbb{P}$ and $\mathbb{Q}$ are posets such that $\mathbb{Q}$ embeds into $\mathbb{P}$, as a forcing poset. Then any $\mathbb{Q}$-generic extension of $V$ can be further extended to a $\mathbb{P}$-generic extension of $V$. If $c$ is a complete classification which remains absolute in all $\mathbb{P}$-generic extensions (as in (3)(b) of Definition~\ref{defn;generic-classification}), then it is also true in all $\mathbb{Q}$-generic extensions.
In conclusion, we may assume that $\mathcal{P}$ is closed under subforcings and forcing equivalence.
\end{remark}

%First, for Remark~\ref{remark: gen class respects reducibility}:
\begin{prop}\label{prop; Apdx borel reduction respects class}
Suppose $E$ and $F$ are analytic equivalence relations on Polish spaces $X$ and $Y$, $f\colon X\to Y$ is a Borel reduction of $E$ to $F$, and $c\colon Y\to I$ is a $\mathcal{P}$-absolute classification of $F$. Then $c\circ f$ is a $\mathcal{P}$-absolute classification of $E$.
\end{prop}
\begin{proof}
Since $f$ is a reduction, $c\circ f$ is a complete classification of $E$ whenever $c$ is a complete classification of $F$. The proposition follows as the statement ``$f$ is a Borel reduction from $E$ to $F$'' is absolute in all forcing extensions.
\end{proof}

Recall the definition of the Friedman-Stanley jump operator $E\mapsto E^+$ from the introduction.
Consider also the following product operation on equivalence relations. Given an equivalence relation $E$ on $X$, $E^\omega$ is defined on $X^\omega$ by $\seqq{x_n}{n<\omega}\mathrel{E^\omega}\seqq{y_n}{n<\omega}$ if $x_n\mathrel{E} y_n$ for all $n<\omega$.
Both these operation preserve ``classifiability by countable structures''. 
Similarly, they preserve generic classifications.
\begin{prop}
Suppose $E$ is an analytic equivalence relation on a Polish space $X$ and $c\colon X\to I$ is a $\mathcal{P}$-absolute classification of $E$.
Then the map $c^\omega\colon X^\omega\to I^\omega$ defined by $c^\omega(\seqq{x_n}{n<\omega})=\seqq{c(x_n)}{n<\omega}$ is a $\mathcal{P}$-absolute classification of $E^\omega$.
\end{prop}
\begin{proof}
The map $c^\omega$ is definable, using a definition of $c$. Parts (1) and (2) of Definition~\ref{defn;generic-classification} follow.
Since $c$ is an $E$-invariant map in any extension, so is $c^\omega$. Finally, in any generic extension by a poset in $\mathcal{P}$, if $c^\omega(\seqq{x_n}{n<\omega})=c^\omega(\seqq{y_n}{n<\omega})$, then $c(x_n)=c(y_n)$ for each $n$, and therefore $x_n\mathrel{E}y_n$ for each $n$, since $c$ is a $\mathcal{P}$-absolute classification.
\end{proof}

\begin{prop}\label{prop: jump is classifiable}
Suppose $E$ is an analytic equivalence relation on a Polish space $X$ and $c\colon X\to I$ is a $\mathcal{P}$-absolute classification of $E$.
Then the map $c^+\colon X^\omega\to\mathcal{P}(I)$ defined by $c^+(\seqq{x_n}{n<\omega})=\set{c(x_n)}{n\in\omega}$ is a $\mathcal{P}$-absolute classification of $E^+$.
\end{prop}
\begin{proof}
Note that $c^+$ can be written as the composition $c^+=s\circ c^\omega$, where $s\colon I^\omega\to\mathcal{P}(I)$ is defined by $s(\seqq{A_n}{n<\omega})=\set{A_n}{n\in\omega}$.
Note that $s$ is definable in an absolute way in all extensions. It follows that $c^+$ is $E^+$-invariant whenever $c^\omega$ is $E^\omega$-invariant, and $c^+$ is a classification of $E^+$ whenever $c^\omega$ is a classification of $E^\omega$. In particular, $c^+$ is $E^+$-invariant in all extensions, and is a classification in any generic extension by a poset in $\mathcal{P}$.
\end{proof}

\begin{prop}\label{prop: Apdx =_R not ordinal classifiable}
Assume that there is some poset $\mathbb{P}\in\mathcal{P}$ such that $\mathbb{P}$ adds a new real and $\mathbb{P}\times\mathbb{P}$ is in $\mathcal{P}$. Then there is no $\mathcal{P}$-absolute classification of $=_\R$ with ordinal classifying invariants. 
\end{prop}
\begin{proof}
Assume for contradiction that $c$ is a $\mathcal{P}$-absolute classification taking ordinal values. Fix $\mathbb{P}$ as above and let $\sigma$ be a $\mathbb{P}$-name for a new real. There is some condition $p\in\mathbb{P}$ and an ordinal $\alpha$ such that $p\force_{\mathbb{P}} c(\sigma)=\check{\alpha}$. We claim that $p$ forces $\sigma$ to be in the ground model, contradicting our assumption. Indeed, given any generic filter $G$ extending $p$, let $H$ be $\mathbb{P}$-generic over $V[G]$ with $p\in H$. Then $c(\sigma[G])=\alpha=c(\sigma[H])$, and so $\sigma[G]=\sigma[H]$, as $c$ is a complete classification in the $\mathbb{P}\times\mathbb{P}$ extension $V[G][H]$. Finally, by Lemma~\ref{lem;mutgen}, $\sigma[G]\in V[G]\cap V[H]=V$.
\end{proof}
The hypothesis in the above proposition holds if $\mathcal{P}$ contains the poset $\P$ for adding a single Cohen real, since $\P\times \P$ and $\P$ are forcing equivalent. In particular, there is no generically absolute classification of $=_\R$ using ordinals.
A similar argument works for random real forcing.
\begin{prop}
Let $\P$ be the poset for adding a single random real in $[0,1]$, with respect to Lebesgue measure. Then there is no $\{\P\}$-absolute classification of $=_\R$ with ordinal classifying invariants. 
\end{prop}
\begin{proof}
Let $\P_2$ be the poset to add a random real in $[0,1]^2$, with the product measure. Then $\P$ and $\P_2$ are forcing equivalent. Moreover, there are $\P_2$-names $\sigma_l,\sigma_r$ for reals in $[0,1]$ so that it is forced that each of $\sigma_l$ and $\sigma_r$ are $\P$-generics, and $V[\sigma_l]\cap V[\sigma_r]=V$. The rest of the argument follows as above.
\end{proof}

\begin{prop}\label{prop; Apdx E0 not P(ord) class}
Assume that $\mathcal{P}$ contains either Cohen forcing or Random real forcing on $2^\omega$ (with the usual product topology and product measure). Then there is no $\mathcal{P}$-absolute classification of $E_0$ with classifying invariants in $2^\alpha$, for any ordinal $\alpha$.
\end{prop}
\begin{proof}
Let $\P\in\mathcal{P}$ be either Cohen or Random real forcing, in $2^\omega$.
Fix an ordinal $\alpha$ and assume for a contradiction that $c\colon 2^\omega\to 2^\alpha$ is a $\mathcal{P}$-absolute complete classification of $E_0$. Let $b\in V$ be the parameter used to define $c$.
Let $x\in 2^\omega$ be $\P$-generic over $V$, and define $A=[x]_{E_0}$.

Both Cohen and Random forcing have the following property:
For any two conditions $p,q\in\P$, there is an automorphism $\pi$ of $\P$, swapping finitely many values of the generic $x\in 2^\omega$, such that $\pi(p)$ is compatible with $q$. In particular, $\pi(\dot{A})=\dot{A}$. It follows that for any formula $\phi$ and parameter $v\in V$, the truth value of $\phi^{V[x]}(A,v)$ is decided in $V$.

Let $y=c(x)$, $y\subset \alpha$. $y$ can be defined in $V[x]$ from $A=[x]_{E_0}$, as follows: for $\zeta<\alpha$, $\zeta\in y$ if and only if there is some $x'\in A$ for which $\zeta\in c(x')$.
The latter can be written as a formula $\phi^{V[x]}(\zeta,A,b)$. It follows that $y$ can be defined in $V$ as the set of $\zeta<\alpha$ for which it is forced that $\phi(\zeta,\dot{A},\check{b})$ holds, and so $y\in V$.

Fix a condition $p\in \P$ forcing that $c(\dot{x})=\check{y}$. Then for any two $\P$-generics $x_1,x_2$ which extend $p$, $x_1\mathrel{E_0}x_2$. We conclude that $E_0$ has a non-meager, or positive measure, equivalence class, a contradiction.
\end{proof}

For each countable ordinal $\alpha$ there is a natural Borel equivalence relation $\cong_\alpha$, with a natural complete classification using hereditarily countable sets in $\mathcal{P}^\alpha(\N)$, the $\alpha$'th-iterated powerset of $\N$ (see \cite{HKL_1998, FS89}). For example, one can take $\cong_0$ to be $=_\N$, and define $\cong_{n+1}$ to be $\cong_n^+$.
Following the approach in \cite[Introduction (E)]{HK96} and \cite{HKL_1998}, an equivalence relation is considered ``classifiable using hereditarily countable sets of rank $\alpha$'' if it is Borel reducible to $\cong_\alpha$.
As ${\cong_{\alpha+1}}\not\leq_B{\cong_\alpha}$, invariants of rank $\alpha+1$ are generally more complex than invariants of rank $\alpha$.

To see that these intuitions are preserved in the context of generically absolute classifications, we would want to show that $\cong_{\alpha+1}$ does not admit a generically absolute classification using hereditarily countable invariants in $\mathcal{P}^\alpha(\N)$.
This follows from the results in \cite{Sha18}.
In fact, similarly to Proposition~\ref{prop; Apdx E0 not P(ord) class} above, this is true without demanding the invariants to be hereditarily countable, and with $\N$ replaced by any ordinal.
\begin{thm}
Fix a countable ordinal $\alpha$, $\beta<\alpha$, any ordinal $\zeta$, and let $I=\mathcal{P}^\beta(\zeta)$. Then $\cong_\alpha$ does not admit a generically absolute classification with invariants in $I$.
\end{thm}
For $\alpha<\omega$, this is proved in \cite[Section 4]{Sha18}, by finding a model of the form $V(A)$, where $A$ is an invariant for $\cong_{n+1}$, so that $V(A)$ cannot be presented as $V(B)$ for any set $B$ in $\mathcal{P}^n(\zeta)$, for any ordinal $\zeta$. The conclusion in \cite{Sha18} is that there is no absolute classification of $\cong_{n+1}$ using invariants in $\mathcal{P}^n(\zeta)$. The same proof shows that there is no such generically absolute classification, since $A$ is constructed as the $\cong_{n+1}$-invariant of a single Cohen-generic real.
The modifications for countable $\alpha\geq\omega$ are explained in \cite[Section 8]{Sha18}.

\bibliographystyle{alpha}
\bibliography{bibliography}

\begin{bibdiv}
\begin{biblist}

\bib{Blass-handbook-2010}{incollection}{
      author={Blass, Andreas},
       title={Combinatorial cardinal characteristics of the continuum},
        date={2010},
   booktitle={Handbook of set theory. {V}ols. 1, 2, 3},
   publisher={Springer, Dordrecht},
       pages={395\ndash 489},
         url={https://doi.org/10.1007/978-1-4020-5764-9_7},
      review={\MR{2768685}},
}

\bib{Conley-Miller-16}{article}{
      author={Conley, Clinton~T.},
      author={Miller, Benjamin~D.},
       title={A bound on measurable chromatic numbers of locally finite {B}orel
  graphs},
        date={2016},
        ISSN={1073-2780},
     journal={Math. Res. Lett.},
      volume={23},
      number={6},
       pages={1633\ndash 1644},
         url={https://doi.org/10.4310/MRL.2016.v23.n6.a3},
      review={\MR{3621100}},
}

\bib{Chichon-Pawlikowski-1986}{article}{
      author={Cicho\'{n}, Jacek},
      author={Pawlikowski, Janusz},
       title={On ideals of subsets of the plane and on {C}ohen reals},
        date={1986},
        ISSN={0022-4812},
     journal={J. Symbolic Logic},
      volume={51},
      number={3},
       pages={560\ndash 569},
         url={https://doi.org/10.2307/2274013},
      review={\MR{853839}},
}

\bib{Cummings2010-handbook}{book}{
      author={Cummings, James},
      editor={Foreman, Matthew},
      editor={Kanamori, Akihiro},
       title={Iterated forcing and elementary embeddings},
   publisher={Springer Netherlands},
     address={Dordrecht},
        date={2010},
        ISBN={978-1-4020-5764-9},
         url={https://doi.org/10.1007/978-1-4020-5764-9_13},
}

\bib{Friedman2000}{article}{
      author={Friedman, Harvey~M.},
       title={Borel and {B}aire reducibility},
        date={2000},
        ISSN={0016-2736},
     journal={Fund. Math.},
      volume={164},
      number={1},
       pages={61\ndash 69},
      review={\MR{1784655}},
}

\bib{FS89}{article}{
      author={Friedman, Harvey},
      author={Stanley, Lee},
       title={A {B}orel reducibility theory for classes of countable
  structures},
        date={1989},
        ISSN={0022-4812},
     journal={J. Symbolic Logic},
      volume={54},
      number={3},
       pages={894\ndash 914},
         url={https://doi.org/10.2307/2274750},
      review={\MR{1011177}},
}

\bib{Gao09}{book}{
      author={Gao, S.},
       title={Invariant descriptive set theory},
   publisher={CRC Press},
        date={2009},
}

\bib{Grigorieff-1975}{article}{
      author={Grigorieff, Serge},
       title={Intermediate submodels and generic extensions in set theory},
        date={1975},
        ISSN={0003-486X},
     journal={Ann. of Math. (2)},
      volume={101},
       pages={447\ndash 490},
         url={https://doi.org/10.2307/1970935},
      review={\MR{373889}},
}

\bib{Halbeisen-book-2017}{book}{
      author={Halbeisen, Lorenz~J.},
       title={Combinatorial set theory},
      series={Springer Monographs in Mathematics},
   publisher={Springer, Cham},
        date={2017},
        ISBN={978-3-319-60230-1; 978-3-319-60231-8},
        note={With a gentle introduction to forcing, Second edition of [
  MR3025440]},
      review={\MR{3751612}},
}

\bib{Hjo00}{book}{
      author={Hjorth, Greg},
       title={Classification and orbit equivalence relations},
      series={Mathematical Surveys and Monographs},
   publisher={American Mathematical Society, Providence, RI},
        date={2000},
      volume={75},
        ISBN={0-8218-2002-8},
         url={https://doi.org/10.1090/surv/075},
      review={\MR{1725642}},
}

\bib{Hjorth-turb-dichotomy-2002}{article}{
      author={Hjorth, Greg},
       title={A dichotomy theorem for turbulence},
        date={2002},
        ISSN={0022-4812},
     journal={J. Symbolic Logic},
      volume={67},
      number={4},
       pages={1520\ndash 1540},
         url={https://doi.org/10.2178/jsl/1190150297},
      review={\MR{1955250}},
}

\bib{Hjorth-Kechris-Ulm-1995}{article}{
      author={Hjorth, Greg},
      author={Kechris, Alexander~S.},
       title={Analytic equivalence relations and {U}lm-type classifications},
        date={1995},
        ISSN={0022-4812},
     journal={J. Symbolic Logic},
      volume={60},
      number={4},
       pages={1273\ndash 1300},
         url={https://doi.org/10.2307/2275888},
      review={\MR{1367210}},
}

\bib{HK96}{article}{
      author={Hjorth, Greg},
      author={Kechris, Alexander~S.},
       title={Borel equivalence relations and classifications of countable
  models},
        date={1996},
        ISSN={0168-0072},
     journal={Ann. Pure Appl. Logic},
      volume={82},
      number={3},
       pages={221\ndash 272},
         url={https://doi.org/10.1016/S0168-0072(96)00006-1},
      review={\MR{1423420}},
}

\bib{Harrington-Kechris-Louvaeu-E0-1990}{article}{
      author={Harrington, L.~A.},
      author={Kechris, A.~S.},
      author={Louveau, A.},
       title={A {G}limm-{E}ffros dichotomy for {B}orel equivalence relations},
        date={1990},
        ISSN={0894-0347},
     journal={J. Amer. Math. Soc.},
      volume={3},
      number={4},
       pages={903\ndash 928},
         url={https://doi.org/10.2307/1990906},
      review={\MR{1057041}},
}

\bib{HKL_1998}{article}{
      author={Hjorth, Greg},
      author={Kechris, Alexander~S.},
      author={Louveau, Alain},
       title={Borel equivalence relations induced by actions of the symmetric
  group},
        date={1998},
     journal={Ann. Pure Appl. Logic},
      volume={92},
       pages={63\ndash 112},
}

\bib{Jech2003}{book}{
      author={Jech, Thomas},
       title={Set theory},
      series={Springer Monographs in Mathematics},
   publisher={Springer-Verlag, Berlin},
        date={2003},
        ISBN={3-540-44085-2},
        note={The third millennium edition, revised and expanded},
      review={\MR{1940513}},
}

\bib{Kechris_1997-classification-problems}{incollection}{
      author={Kechris, Alexander~S.},
       title={Actions of {P}olish groups and classification problems},
        date={2002},
   booktitle={Analysis and logic ({M}ons, 1997)},
      series={London Math. Soc. Lecture Note Ser.},
      volume={262},
   publisher={Cambridge Univ. Press, Cambridge},
       pages={115\ndash 187},
      review={\MR{1967835}},
}

\bib{Kechris-ctbl-sections-1992}{article}{
      author={Kechris, Alexander~S.},
       title={Countable sections for locally compact group actions},
        date={1992},
        ISSN={0143-3857},
     journal={Ergodic Theory Dynam. Systems},
      volume={12},
      number={2},
       pages={283\ndash 295},
         url={https://doi.org/10.1017/S0143385700006751},
      review={\MR{1176624}},
}

\bib{Kechris-DST-1995}{book}{
      author={Kechris, Alexander~S.},
       title={Classical descriptive set theory},
      series={Graduate Texts in Mathematics},
   publisher={Springer-Verlag, New York},
        date={1995},
      volume={156},
        ISBN={0-387-94374-9},
         url={https://doi.org/10.1007/978-1-4612-4190-4},
      review={\MR{1321597}},
}

\bib{Kechris-Louveau-1997}{article}{
      author={Kechris, Alexander~S.},
      author={Louveau, Alain},
       title={The classification of hypersmooth {B}orel equivalence relations},
        date={1997},
        ISSN={0894-0347},
     journal={J. Amer. Math. Soc.},
      volume={10},
      number={1},
       pages={215\ndash 242},
         url={https://doi.org/10.1090/S0894-0347-97-00221-X},
      review={\MR{1396895}},
}

\bib{Kanovei-Reeken}{article}{
      author={{Kanovej}, V.~G.},
      author={{Reeken}, M.},
       title={{Some new results on the Borel irreducibility of equivalence
  relations.}},
    language={English},
        date={2003},
        ISSN={1064-5632},
     journal={{Izv. Math.}},
      volume={67},
      number={1},
       pages={55\ndash 76},
}

\bib{Kanovei-Sabok-Zapletal-2013}{book}{
      author={Kanovei, Vladimir},
      author={Sabok, Marcin},
      author={Zapletal, Jind\v{r}ich},
       title={Canonical {R}amsey theory on {P}olish spaces},
      series={Cambridge Tracts in Mathematics},
   publisher={Cambridge University Press, Cambridge},
        date={2013},
      volume={202},
        ISBN={978-1-107-02685-8},
         url={https://doi.org/10.1017/CBO9781139208666},
      review={\MR{3135065}},
}

\bib{Larson_Zapletal_2020}{book}{
      author={Larson, P.},
      author={Zapletal, J.},
       title={Geometric set theory},
      series={Mathematical Surveys and Monographs},
        date={2020},
      volume={248},
}

\bib{Sha18}{article}{
      author={Shani, Assaf},
       title={Borel reducibility and symmetric models},
        date={2021},
        ISSN={0002-9947},
     journal={Trans. Amer. Math. Soc.},
      volume={374},
      number={1},
       pages={453\ndash 485},
         url={https://doi.org/10.1090/tran/8250},
      review={\MR{4188189}},
}

\bib{Solecki-2002-composants}{article}{
      author={Solecki, Slawomir},
       title={The space of composants of an indecomposable continuum},
        date={2002},
        ISSN={0001-8708},
     journal={Adv. Math.},
      volume={166},
      number={2},
       pages={149\ndash 192},
         url={https://doi.org/10.1006/aima.2001.2002},
      review={\MR{1895561}},
}

\bib{Thomas-2008-quasi-iso}{article}{
      author={Thomas, Simon},
       title={On the complexity of the quasi-isometry and virtual isomorphism
  problems for finitely generated groups},
        date={2008},
        ISSN={1661-7207},
     journal={Groups Geom. Dyn.},
      volume={2},
      number={2},
       pages={281\ndash 307},
         url={https://doi.org/10.4171/GGD/41},
      review={\MR{2393184}},
}

\bib{Ulrich-Rast-Laskowski-2017}{article}{
      author={Ulrich, Douglas},
      author={Rast, Richard},
      author={Laskowski, Michael~C.},
       title={Borel complexity and potential canonical {S}cott sentences},
        date={2017},
        ISSN={0016-2736},
     journal={Fund. Math.},
      volume={239},
      number={2},
       pages={101\ndash 147},
         url={https://doi.org/10.4064/fm326-11-2016},
      review={\MR{3681584}},
}

\bib{Zapletal-idealized-2008}{book}{
      author={Zapletal, Jind\v{r}ich},
       title={Forcing idealized},
      series={Cambridge Tracts in Mathematics},
   publisher={Cambridge University Press, Cambridge},
        date={2008},
      volume={174},
        ISBN={978-0-521-87426-7},
         url={https://doi.org/10.1017/CBO9780511542732},
      review={\MR{2391923}},
}

\end{biblist}
\end{bibdiv}

\end{document}